\documentclass[reqno]{amsart}
\usepackage{amsmath,amsthm,amscd,amssymb,amsfonts, amsbsy}
\usepackage{latexsym}
\usepackage{pxfonts}
\usepackage{enumerate}
\usepackage{verbatim}
\usepackage[usenames,dvipsnames]{color}

\numberwithin{equation}{section}

\theoremstyle{plain}
\newtheorem{theorem}[equation]{Theorem}
\newtheorem{lemma}[equation]{Lemma}
\newtheorem{corollary}[equation]{Corollary}

\theoremstyle{definition}

\newtheorem*{acknowledgment}{Acknowledgment}

\theoremstyle{remark}
\newtheorem{remark}[equation]{Remark}

\newcommand{\dv}{\operatorname{div}}

\newcommand{\mysection}[1]{\section{#1}
\setcounter{equation}{0}}

\newcommand{\bR}{\mathbb R}
\newcommand{\bH}{\mathbb H}
\newcommand{\bP}{\mathbb P}
\newcommand{\bQ}{\mathbb Q}
\newcommand{\bN}{\mathbb N}

\newcommand{\bZ}{\mathbb Z}

\newcommand\cH{\mathcal{H}}

\newcommand\cL{\mathcal{L}}
\newcommand\cP{\mathcal{P}}

\newcommand\hd{\hat{\delta}}
\newcommand\hp{\hat{p}}
\newcommand\tD{\tilde{D}}
\newcommand\tv{\tilde{v}}
\newcommand\tu{\tilde{u}}
\newcommand\tf{\tilde{f}}

\newcommand\tT{\tilde{T}}

\newcommand{\set}[1]{\left\{#1\right\}}
\newcommand{\norm}[1]{\lVert#1\rVert}
\newcommand{\Norm}[1]{\left\lVert#1\right\rVert}
\newcommand{\Abs}[1]{\left\lvert#1\right\rvert}
\newcommand{\abs}[1]{\lvert#1\rvert}

\renewcommand{\epsilon}{\varepsilon}
\renewcommand{\vec}[1]{\boldsymbol{#1}}
\renewcommand{\qedsymbol}{$\blacksquare$}

\begin{document}
\title[Partial Schauder estimates]
{Partial Schauder estimates for second-order elliptic and parabolic equations: a revisit}

\author[H. Dong]{Hongjie Dong}
\address[H. Dong]{Division of Applied Mathematics, Brown University,
182 George Street, Providence, RI 02912, United States of America}
\email{Hongjie\_Dong@brown.edu}
\thanks{H. Dong was partially supported by the NSF under agreements DMS-1056737 and DMS-1600593.}

\author[S. Kim]{Seick Kim}
\address[S. Kim]{Department of Mathematics, Yonsei University, 50 Yonsei-ro, Seodaemun-gu, Seoul 03722, Republic of Korea}
\email{kimseick@yonsei.ac.kr}
\thanks{S. Kim is partially supported by NRF-2014R1A1A2056839.}



\begin{abstract}
Under various conditions, we establish Schauder estimates for both divergence and non-divergence form second-order elliptic and parabolic equations involving H\"older semi-norms not with respect to all, but only with respect to some of the independent variables.
A novelty of our results is that the coefficients are allowed to be merely measurable with respect to the other independent variables.
\end{abstract}

\maketitle

\mysection{Introduction}

The classical Schauder theory was established by J. Schauder about eighty years ago and since then plays an important role in  the existence theory for linear and non-linear elliptic and parabolic equations.
Roughly speaking, the Schauder theory for second-order elliptic equations in non-divergence (or divergence) form says that if all the coefficients and data are H\"older continuous in all variables, then the same holds for the second (or the first, respectively) derivatives of the solution.
For second-order parabolic equations in non-divergence (or divergence) form, the Schauder theory reads that  if the coefficients and data are H\"older continuous in both space and time variables, then the same holds for the second spatial derivatives and the first time derivative of the solution (or the solution itself and the first spatial derivatives, respectively).
Such results were proved both in the interior of the domain and near the boundary with appropriate boundary conditions, as well as for higher-order equations and systems. See, for instance, \cite{ADN64}.

For higher-order elliptic equations with smooth coefficients, by using the potential theory as in \cite{ADN64}, P. Fife \cite{Fife} established certain Schauder estimates involving H\"older semi-norms not with respect to all, but only with respect to some of the independent variables.
It was also observed by B. Knerr \cite{Knerr} and G. Lieberman \cite{Lieb92} that, for second-order parabolic equations in both divergence and non-divergence form, the regularity assumption on the coefficients and data with respect to the time variable can sometimes be dropped, which is also known as {\em intermediate} Schauder theory.
The proofs in \cite{Knerr} are based on the maximum principle, while in \cite{Lieb92} both the maximum principle and the Campanato's approach are used. See \cite{Lorenzi,KrPr} and references therein for other more recent results in this direction.
These are some earlier work on what we shall hereafter refer to \textit{partial Schauder estimates}, which is the subject of the current paper. Partial Schauder estimates have attracted many attentions due to their important applications, for instance, in problems arising from linearly elastic laminates and composite materials (cf. \cite{CKV,LV00,LN03}).

To be more precise, we first fix some related notation.
Let $x=(x^1,\ldots, x^d)$ be a point in $\bR^d$, with $d\ge 2$, and $q$ be an integer such that $1\le q <d$.
We distinguish the first $q$ coordinates of $x$ from the rest and write $x=(x',x'')$, where
\[
x'=(x^1,\ldots,x^q)\quad \text{and}\quad x''=(x^{q+1},\ldots,x^d).
\]
Roughly speaking, $x'$ denotes ``good'' coordinate variables while $x''$ represents ``bad'' coordinate variables.
For a function $u$ on a domain $\Omega\subset\bR^d$, naturally we define a \textit{partial H\"older semi-norm} with respect to $x'$ as
\[
[u]_{x',\delta;\,\Omega}:=\sup_{\substack{(x',x''), \,(y',x'') \in \Omega \\x' \neq  y'}} \frac{\abs{u(x',x'')-u(y',x'')}}{\abs{x'-y'}^\delta}.
\]
Throughout this article, we assume $0<\delta<1$ unless explicitly otherwise stated.

Let us mention some recent work on partial Schauder estimates in \cite{DK11,TW,Dong12,JLW14}, which are closely related the current paper. In \cite{DK11}, we considered both divergence and non-divergence form second-order scalar elliptic and parabolic equations.
Among other results, we proved that if the coefficients are independent of $x'$ and the data are H\"older continuous with respect to $x'$, then derivatives of solutions with respect to $x'$ are H\"older continuous in $x'$.
By using a different method, G. Tian and X.-J. Wang \cite{TW} proved similar results for non-divergence form elliptic equations with data Dini continuous in some variables.
Under certain conditions, their results also extend to second-order fully nonlinear equations.
We note that for non-divergence form equations, in both \cite{DK11} and \cite{TW} the coefficients are assumed to be continuous in $x$, even though the estimates are independent of the moduli of continuity with respect to $x$.
In \cite{Dong12}, the first named author studied second-order divergence form elliptic and parabolic systems as well as non-divergence scalar equations, with coefficients and data H\"older or Dini continuous in the time variable and all but one spatial variable, i.e., $q=d-1$.
In particular, it is proved that for non-divergence form equations if the coefficients and data are Dini continuous in $z':=(t,x')$ and merely measurable in $x^d$, then any solution $u$ is $C^1$ in $t$, $C^{1,1}$ in $x$, and $u_t$ and $D_{xx'} u$ are continuous.
Under the stronger condition that the coefficients and data are H\"older continuous in $z'$, $u_t$ and $D_{xx'} u$ are H\"older continuous in all variables.
In a very recent paper \cite{JLW14}, Y. Jin, D. Li, and X.-J. Wang obtained the following results for non-divergence form elliptic equations: If the coefficients are independent of a direction $\xi$ and the data is analytic in $\xi$, then any strong solution is analytic in $\xi$; if the leading coefficients are continuous, and the coefficients and data are analytic in a direction $\xi$, then any strong solution is analytic in $\xi$; if the leading coefficients are continuous, and the coefficients and data are H\"older continuous in $\xi$, then for any strong solution $u$, $D_{x\xi} u$ is H\"older continuous.

The main objective of this paper is to study the regularity of solutions for both divergence and non-divergence form elliptic and parabolic equations when the coefficients are merely {\em measurable} in the ``bad'' directions, which are allowed to be more than one.
We give a brief account of our main results as follows.
In the elliptic case, we assume that data is H\"older continuous in $x'$ and treat the following three classes of coefficients:
\begin{enumerate}[1.]
\item
The coefficients are independent of $x'$ and with no regularity assumption with respect to $x''$.
\item
The coefficients are H\"older continuous in $x'$ and with no regularity assumption with respect to $x''$.
\item
The coefficients are uniformly continuous in $(x^1,\ldots,x^{d-1})$, merely measurable in $x^d$, and H\"older continuous in $x'$.
\end{enumerate}
In the first case, we show that for any $W^2_d$ strong solution (or  $W^1_2$ weak solution) $u$ of non-divergence (or divergence) form equations, $D_{x'}^2 u$ (or $D_{x'} u$, respectively) are H\"older continuous in $x'$.
This result improves the aforementioned results in \cite{DK11,TW} by removing the continuity condition with respect to $x$ (see Theorem \ref{thm1}).
The main novelty of our paper lies in the second case, in which we prove that for any $W^2_p$ strong solution (or any $W^1_p$ weak solution) with a sufficient large $p$, $D_{x'}^2 u$ (or $D_{x'} u$) are H\"older continuous in {\em all} the variables (see Theorem \ref{thm4m}).
In the particular case when $q=d-2$, one can actually estimate the H\"older norm of $D_{xx'}u$ for non-divergence form equations, and the H\"older norm of $D_{x'}u$ for divergence form strongly elliptic systems (see Theorem \ref{thm4mn}).
In the last case, we show that for any strong solution (or any  weak solution), $D_{xx'} u$ (or $D_{x'} u$) are H\"older continuous in all the variables (see Theorem \ref{thm4}).
Analogous results for parabolic equations are also established.

As mentioned before, for non-divergence elliptic equations, an estimate similar to Theorem \ref{thm4} was recently proved in \cite{JLW14}.
Compared to \cite{JLW14}, our proof is technically different and we also obtain a sharper H\"older exponent.

Our proofs are all based on the Campanato's approach, but with various techniques in the three different cases.
They work equally well for both divergence and non-divergence elliptic and parabolic equations.
Let us give a short descriptions of the proofs in the elliptic case.
For Theorem \ref{thm1}, we mainly follow the outline of the argument in \cite{DK11}, which in turn adopts an idea of decomposition by M. V. Safonov and the mollification method of N. Trudinger.
In the proof, we emphasize how to use an approximation argument to remove the continuity condition and also how to localize the estimates by using an iteration argument.
The main idea of the proof of Theorem \ref{thm4m} is to apply an $L_\varepsilon$ version of Campanato's characterization of H\"older continuous functions (cf. Lemma \ref{lemma6.14}).
To the best of our knowledge, such application is new, as the Campanato's approach is usually used in the $L_p$ setting for $p\ge 1$ (mostly in the case $p=2$ or $p=\infty$).
The proof also relies on the Krylov--Safonov estimate, the De Giorgi--Nash--Moser estimate, and an $W^2_\varepsilon$ estimate due to F. Lin \cite{Li86}.
In the proof of Theorem \ref{thm4mn}, we exploit the reverse H\"older's inequality for elliptic systems and a recent result in \cite{DoKr10} by the first named author and N. V. Krylov about the $W^2_{2+\varepsilon}$ estimates for non-divergence form elliptic equations with coefficients measurable in two directions.
Finally, for Theorem \ref{thm4} we appeal to some recent work in \cite{KK07, DK09, Dong12b} on $W^2_p$ (or $W^1_p$) estimates for elliptic equations with coefficients measurable in one direction, from which we obtain an interior $C^{1,\alpha}$ (or $C^\alpha$) estimate for solutions to homogeneous equations.

The organization of this paper is as follows.
We state our main results for elliptic equations in Section \ref{sec:m}.
Section \ref{sec:e} is devoted to the proofs of these results.
The main results for parabolic equations are stated in Section \ref{sec:p} and their proofs are given in Section \ref{sec:pp}, where we also use a special type of interpolation inequalities proved in the Appendix for parabolic H\"older semi-norms, which might be of independent interest.
\newpage

\mysection{Main Results for elliptic equations}		\label{sec:m}

We consider elliptic operators in non-divergence form
\[
Lu:=a^{ij}(x)D_{ij}u,
\]
and in divergence form
\[
\cL u:=D_i(a^{ij}(x)D_j u),
\]
where the coefficients $a^{ij}(x)$ are measurable functions on $\bR^d$ satisfying the uniform ellipticity condition
\begin{equation}		\label{elliptic}
\nu \abs{\xi}^2\le a^{ij}(x) \xi^i\xi^j\le \nu^{-1} \abs{\xi}^2,\quad\forall x\in\bR^{d},\,\, \xi\in \bR^d,
\end{equation}
for some constant $\nu\in (0,1]$.
We assume the symmetry of the coefficients (i.e., $a^{ij}=a^{ji}$) for the operators $L$ in non-divergence form.
For the operators $\cL$ in divergence form, we do not require coefficients to be symmetric but instead assume that they are bounded; to avoid introducing a new constant, we simply assume that $\sum_{i,j=1}^d \abs{a^{ij}(x)}^2 \le \nu^{-2}$ for all $x \in \bR^d$ instead of the second inequality in \eqref{elliptic}.

For  $k=1,2,\ldots$, we set
\[
 [u]_{x',k+\delta;\,\Omega}=[D_{x'}^k u]_{x',\delta;\,\Omega}=\max_{\alpha\in\bZ_+^q,\, \abs{\alpha}=k}[\tD^\alpha u]_{x',\delta; \,\Omega},
\]
where we used the usual multi-index notation and $\tD^\alpha:=D_1^{\alpha_1} \cdots D_q^{\alpha_q}$.
We also use the notation
\[
\abs{u}_{0;\, \Omega}= \sup_{\Omega}\, \abs{u}.
\]
When $\Omega=\bR^d$, we will drop the reference to the domain and simply write
\[
[u]_{x', \delta}=[u]_{x',\delta; \,\bR^d},\;\text{ etc.}
\]
We denote $C^{k}_{x'}(\Omega)$ the set of all bounded measurable functions $u$ on $\Omega$ whose derivatives $\tilde{D}^\alpha u$ for $\alpha\in \bZ_+^q$ with $\abs{\alpha}\le k$ are continuous and bounded in $\Omega$.
We denote by $C^{k+\delta}_{x'}(\Omega)$ the set of all functions $u\in C^{k}_{x'}(\Omega)$ for which the partial H\"older semi-norm $[u]_{x', k+\delta; \,\Omega}$ is finite.
We use the notation $W^k_p(\Omega)$, $k=1,2,\ldots$, for the Sobolev spaces in $\Omega$.

For $p\in (1,\infty)$, we say that $u\in W^2_{p;\, loc}(\Omega)$ is a strong solution of $Lu=f$ in $\Omega$ if $u$ satisfies the equation $Lu=f$ a.e. in $\Omega$.
In the first theorem below, we assume that $a^{ij}$ are independent of $x'$ and merely measurable in $x''$.
We denote by $B_r(x_0)$ the Euclidean ball with radius $r$ centered at $x_0$.
When the center is the origin, we simply write $B_r$ for $B_r(0)$.

\begin{theorem}
                                    \label{thm1}
Assume that  $a=[a^{ij}]$ are independent of $x'$.
\begin{enumerate}[(i)]
\item
Let $u\in W^2_{d;\,loc}$ be a bounded strong solution of the equation
\[
L u=f \quad \text{in }\,\bR^d,
\]
where $f \in C^\delta_{x'}$.
Then $u\in C^{2+\delta}_{x'}$ and there is a constant $N=N(d,q,\nu, \delta)$ such that
\begin{equation}
                                            \label{eq3.58}
[u]_{x',2+\delta} \le N[f]_{x',\delta}.
\end{equation}

\item
Let $u\in W^2_{d;\,loc}(B_1)$ be a bounded strong solution of the equation
\[
L u=f \quad \text{in }\,B_1,
\]
where $f \in C^\delta_{x'}(B_1)$. Then $u\in C^{2+\delta}_{x'}(B_{1/2})$ and there is a constant $N=N(d,q,\nu, \delta)$ such that
\begin{equation}
                                            \label{eq3.58b}
[u]_{x',2+\delta; \,B_{1/2}}\le N \left([f]_{x',\delta; \,B_1}+ \abs{u}_{0; \,B_1} \right).
\end{equation}

\item
Let $u \in W^1_2(B_1)$ be bounded weak solution of the equation
\[
\cL u = \dv \vec f\quad \text{in }\, B_1
\]
where $\vec f = (f^1, \ldots, f^d) \in C^\delta_{x'}(B_1)$. Then $u\in C^{1+\delta}_{x'}(B_{1/2})$ and there is a constant $N=N(d,q,\nu, \delta)$ such that
\[
[u]_{x',1+\delta; \, B_{1/2}}\le N \left([\vec f]_{x',\delta; \,B_1}+ \abs{u}_{0; \,B_1} \right).
\]
\end{enumerate}
\end{theorem}

\begin{remark}
Under the additional assumption that $a^{ij}$ are uniformly continuous with respect to $x''$, results similar to Theorem \ref{thm1} were proved in \cite{DK11} and \cite{TW}.
\end{remark}

In the next theorem, we assume that $a^{ij}$ are H\"older continuous in $x'$ and merely measurable in $x''$.

\begin{theorem}
                                    \label{thm4m}
Let $\delta\in (0,1]$ and $p\in (d,\infty)$ be such that $\delta-d/p>0$.
Assume that $a=[a^{ij}]$ are $\delta$-H\"older continuous in $x'$ and merely measurable in $x''$.
Then, there exists a constant $\delta_0=\delta_0(d,\nu) > 0$ such that the following assertions hold with any $\hd \in (0, \delta_0)$ satisfying $\hd \le \delta-d/p$.
\begin{enumerate}[(i)]
\item
Let $u\in W^2_{p}(B_1)$ be a strong solution of the equation
\[
L u=f \quad \text{in }\,B_1,
\]
where $f \in  C^{\hd}_{x'}(B_1)$. Then $D^2_{x'}u\in C^{\hd}(B_{1/2})$ and there is a constant $N=N(d, q, \nu, p, \delta)$ such that
\begin{equation}
                                            \label{eq3.59}
[D^2_{x'}u]_{\hd; \,B_{1/2}}\le N \left([f]_{x',\hd;\, B_1}+ \left(1+ [a]_{x', \delta; \,B_1}\right) \norm{D^2 u}_{L_p(B_1)} \right).
\end{equation}

\item
Let $u\in W^1_{p}(B_1)$ be a weak solution of the equation
\[
\cL u=\dv \vec f \quad \text{in }\,B_1,
\]
where $\vec f= (f^1,\ldots,f^d) \in C^{\hd}_{x'}(B_1)$.
Then $D_{x'}u\in C^{\hd}(B_{1/2})$ and there is a constant $N=N(d, q, \nu, p, \delta)$ such that
\begin{equation}
                                            \label{eq4.00}
[D_{x'}u]_{\hd; \,B_{1/2}}\le N \left([\vec f]_{x',\hd; \,B_1}+\left(1+ [a]_{x', \delta; \,B_1}\right) \norm{Du}_{L_p(B_1)} \right).
\end{equation}
\end{enumerate}
\end{theorem}

\begin{remark}
For non-divergence elliptic equations, an estimate similar to \eqref{eq3.59} was obtained in \cite{TW} under the assumption that $a^{ij}$ are independent of $x'$ and continuous in $x''$.
\end{remark}

Our next result is regarding the special case when $q=d-2$.
\begin{theorem}
                                    \label{thm4mn}
Assume $d\ge 3$ and $q=d-2$ so that $x''=(x^{d-1}, x^d)$.
Let $\delta\in (0,1]$ and $p\in (d,\infty)$ be such that $\delta-d/p>0$.
Assume $a=[a^{ij}]$ are $\delta$-H\"older continuous in $x'$ but merely measurable in $x''$.
Then, there exists a constant $\delta_0=\delta_0(d,\nu) > 0$ such that the following assertions hold with any $\hd \in (0, \delta_0)$ satisfying $\hd \le \delta-d/p$.
\begin{enumerate}[(i)]
\item
Let $u\in W^2_{p}(B_1)$ be a strong solution of the equation
\[
L u=f \quad \text{in }\,B_1,
\]
where $f \in C^{\hd}_{x'}(B_1)$. Then $DD_{x'}u\in C^{\hd}(B_{1/2})$ and there is a constant $N=N(d, \nu, p, \delta)$ such that
\begin{equation}
                                            \label{eq3.59mn}
[DD_{x'}u]_{\hd; \,B_{1/2}}\le N \left([f]_{x',\hd; \,B_1}+ \left(1+ [a]_{x', \delta; \,B_1}\right) \norm{D^2 u}_{L_p(B_1)} \right).
\end{equation}

\item
The statement of Theorem~\ref{thm4m} (ii) still holds for strongly elliptic systems.

\end{enumerate}
\end{theorem}

We obtain better regularity for $u$ when the coefficients are assumed to be uniformly continuous in $(x^1,\ldots,x^{d-1})$, i.e., all but one independent variable.
\begin{theorem}
                                    \label{thm4}
Let $\delta\in (0,1]$ and $p\in (d,\infty)$ be such that $\hd:=\delta-d/p>0$.
Assume that $a=[a^{ij}]$ are uniformly continuous in $(x^1,\ldots,x^{d-1})$, merely measurable in $x^d$, and $\delta$-H\"older continuous in $x'$.
Let $\omega_{a}$ denote a modulus of continuity of $a=[a^{ij}]$ with respect to $(x^1,\ldots, x^{d-1})$.
\begin{enumerate}[(i)]
\item
Let $u\in W^2_{p}(B_1)$ be a strong solution of the equation
\[
L u=f \quad \text{in }\,B_1,
\]
where $f \in C^{\hd}_{x'}(B_1)$. Then $DD_{x'}u\in C^{\hd}(B_{1/2})$ and there is a constant $N$ depending only on $d$, $q$,  $\nu$, $p$,  $\delta$, and $\omega_a$ such that
\begin{equation}
                                            \label{eq4.11}
[DD_{x'}u]_{\hd; \,B_{1/2}}\le N \left([f]_{x',\hd; \,B_1}+ \left(1+ [a]_{x', \delta; \,B_1}\right) \norm{D^2 u}_{L_p(B_1)}\right).
\end{equation}

\item
Let $u\in W^1_{p}(B_1)$ be a weak solution of the equation
\[
\cL u=\dv \vec f \quad \text{in }\,B_1,
\]
where $\vec f =(f^1,\ldots, f^d) \in C^{\hd}_{x'}(B_1)$.
Then $D_{x'}u\in C^{\hd}(B_{1/2})$ and there is a constant $N$ depending only on $d$, $q$, $\nu$, $p$, $\delta$, and $\omega_a$ such that
\begin{equation}
                                            \label{eq4.12}
[D_{x'}u]_{\hd; \,B_{1/2}}\le N \left([\vec f]_{x',\hd; \,B_1}+ \left(1+[a]_{x', \delta; \,B_1}\right) \norm{Du}_{L_p(B_1)}\right).
\end{equation}
\end{enumerate}
\end{theorem}

\begin{remark}
By the interior $L_p$ estimates established in \cite{KK07} and \cite{DK09}, in Theorem \ref{thm4} (i) we may assume that $u\in W^2_{\hat p}(B_1)$ for some small $\hat p>1$, and the term $\norm{D^2u}_{L_p(B_1)}$ on the right-hand side of \eqref{eq4.11} can be replaced by the sum of a weaker norm of $u$ and the $L_p$ norm of $f$; for example, we would obtain an estimate like
\[
[DD_{x'}u]_{\hd; \,B_{1/2}}\le N \left([f]_{x',\hd; \,B_1}+\norm{f}_{L_p(B_1)} + \norm{u}_{L_{\hat p}(B_1)} \right),
\]
where $N$ depends only on $d$, $q$, $\nu$, $p$, $\hat p$, $\delta$, $[a]_{x', \delta; \, B_1}$ and $\omega_a$.
Similarly, in Theorem \ref{thm4} (ii) we may assume that $u\in W^1_{\hat p}(B_1)$ for some small $\hat p>1$, and the term $\norm{Du}_{L_p(B_1)}$ on the right-hand side of \eqref{eq4.12} can be replaced by the sum of a weaker norm of $u$ and the $L_{p}$ norm of $\vec f$.
\end{remark}

\begin{remark}
For non-divergence elliptic equations, an estimate similar to \eqref{eq4.11} was recently shown in \cite{JLW14} by using a different proof.
It should be pointed out that in \cite{JLW14} it is assumed that $\hd$ is strictly less than $\delta-d/p$.
\end{remark}

\begin{remark}
We only consider operators without lower-order terms for the sake of simplicity of the presentation. In Section \ref{sec3.5}, we will illustrate how to extend our results to equations {\em with} lower-order terms.
\end{remark}

\mysection{The proofs: Elliptic estimates}		\label{sec:e}

\subsection{Proof of Theorem~\ref{thm1}}		\label{sec3.1}
We prove the theorem in essence by following M. V. Safonov's idea of applying equivalent norms and representing solutions as sums of ``small'' and smooth functions.
However, his argument as reproduced in the proof of \cite[Theorem 3.4.1]{Kr96} is not directly applicable in our case by several technical reasons and to get around this difficulty we also make use of the mollification method of Trudinger \cite{Trudinger}.

For a function $v$ defined on $\bR^d$ and $\epsilon>0$, we define a \textit{partial mollification} of $v$ with respect to the first $q$ coordinates $x'$ as
\begin{equation}		
				\label{eq:pamol}
\tv^\epsilon(x',x''):=\frac{1}{\epsilon^q}\int_{\bR^q}v(y',x'') \zeta\left(\frac{x'-y'}{\epsilon}\right)\,dy'=
\int_{\bR^q}v(x'-\epsilon y',x'')\zeta(y')\,dy',
\end{equation}
where $\zeta(x^1,\ldots,x^q)=\prod_{i=1}^q \eta(x^i)$ and $\eta=\eta(t)$ is a smooth function on $\bR$ with a compact support in $(-1,1)$ satisfying $\int \eta=1$, $\int t \eta\,dt=0$, and $\int t^2 \eta \,dt=0$.
We assume further that the support of $\eta$ is chosen so small that $\zeta \in C^\infty_c(B_1)$.
Then, by virtue of Taylor's formula, it is not hard to prove the following lemma for partial mollifications (see, e.g., \cite[Chapter 3]{Kr96}).

\begin{lemma}
                                    \label{lem12.25}
Let $\epsilon>0$ and $x_0\in \bR^d$.
\begin{enumerate}[(i)]
\item
Suppose $v\in C^{\delta}_{x'}(B_{\varepsilon}(x_0))$. Then we have
\[
\epsilon^{1-\delta}\abs{D_{x'} \tv^\epsilon(x_0)}+\epsilon^{2-\delta} \abs{D^{2}_{x'} \tv^\epsilon(x_0)} \le N(d,q,\delta,\eta) [v]_{x',\delta;\, B_{\varepsilon}(x_0)},
\]
\item
Suppose $v\in C_{x'}^{k+\delta}(B_\varepsilon(x_0))\,\,(k=0,1,2)$. Then we have
\[
\abs{v(x_0)-\tv^\epsilon(x_0)} \le N(d,q,\delta,\eta)\epsilon^{k+\delta}[v]_{x',k+\delta; \,B_{\varepsilon}(x_0)}.
\]
\end{enumerate}
\end{lemma}

For $k=0,1,2,\ldots$, denote by $\tilde \bP_k$, which we shall refer to as the set of all $k$th-order partial polynomials in $x'$, the set of all functions $p=p(x',x'')$ on $\bR^d$ such that  $p(x',x'')$ is a polynomial of $x'\in \bR^q$ of degree at most $k$ for any $x''$.
We will also use the following notation for a partial Taylor's polynomial of order $k$ with respect to $x'$ of a function $v$ at a point $x_0'$:
\[
\tT^k_{x_0'}v(x',x''):=\sum_{\alpha\in \bZ_+^q,\, \abs{\alpha} \le k} \frac{1}{\alpha!} (x'-x_0')^\alpha
\tD^\alpha v(x_0',x'').
\]

First we prove assertion (i) of the theorem and derive an a priori estimate for $u$ assuming that $u\in C_{x'}^{2+\delta}(\bR^d)$.
By mollification, we can find a sequence of coefficients $a_n=[a^{ij}_n]$, which are continuous, independent of $x'$, satisfy \eqref{elliptic}, and $a_n \to a$ a.e. as $n\to \infty$. Let $L_n$ be the corresponding operator with $a_n$ in place of $a$.
Then we have
\[
L_n u=f_n,\quad \text{where}\quad f_n=f+(a^{ij}_n-a^{ij})D_{ij}u.
\]
Let $\kappa>2$ be a number to be chosen later.
Since $a^{ij}_n$ are independent of $x'$, we have for any $r>0$,
\[
L_n \tu^{\kappa r}=\tf_{n}^{\kappa r}.
\]
Let $B_r=B_r(x_0)$, where $x_0$ is a point in $\bR^d$, and let $w\,(=w_n)\in W^2_{d;\,loc}(B_{\kappa r})\cap C^0(\overline B_{\kappa r})$ be a unique solution of the Dirichlet problem (see \cite[Corollary 9.18]{GT})
\begin{equation}							 \label{eq1.08}
\left\{
  \begin{aligned}
    L_n w = 0 \quad & \hbox{in $\;B_{\kappa r}$,} \\
    w=u-\tu^{\kappa r} \quad & \hbox{on $\;\partial B_{\kappa r}$.}
  \end{aligned}
\right.
\end{equation}
By the ABP maximum principle and Lemma \ref{lem12.25} (ii), we obtain
\begin{equation}
                                                        \label{eq1.27}
\sup_{B_{\kappa r}}\, \abs{w}= \sup_{\partial B_{\kappa r}}\, \abs{w} \le N(\kappa r)^{2+\delta} [u]_{x',2+\delta;\, B_{2\kappa r}}.
\end{equation}
It follows from the theory of Krylov and Safonov that $w$ is locally H\"older continuous in $B_{\kappa r}$ with a H\"older exponent $\delta_0=\delta_0(d,\nu)\in (0,1)$.
Since $a^{ij}$ are independent of $x'$, it is reasonable to expect from \eqref{eq1.08} a better interior estimate for $w$ with respect to $x'$.
Indeed, by using a technique of the finite difference quotients and bootstrapping (see, e.g., \cite[\S 5.3]{CaCa95}), one easily gets from the H\"older estimates of  Krylov and Safonov that, for any integer $j \ge 1$,
\begin{equation}
                                            \label{eq16.41}
\abs{D^j_{x'}w}_{0; \,B_{\kappa r/2}}\le (\kappa r)^{-j} N(j,d,q,\nu) \,\abs{w}_{0;\,B_{\kappa r}},
\end{equation}
where we used notation
\[
\abs{D_{x'}^j w}_{0; \,B_r}=\max_{\alpha\in\bZ_+^q,\; \abs{\alpha}=j} \abs{\tilde{D}^\alpha w}_{0; \,B_r} \quad\text{and}\quad \abs{w}_{0; \,B_r}=\sup_{B_r} \,\abs{w}.
\]
In particular, with $j=3$, we get
\begin{align}
                                                    \label{Eq1.48}
\abs{w- \tT^2_{x_0'}w}_{0;\,B_r} &\le N r^3 \abs{D^3_{x'} w}_{0;\,B_r}\le Nr^3 \abs{D^3_{x'} w}_{0;\,B_{\kappa r/2}}\\
\nonumber
&\le N \kappa^{-3} \abs{w}_{0;\,B_{\kappa r}}\le N \kappa^{\delta-1} r^{2+\delta}[u]_{x',2+\delta; \,B_{2\kappa r}},
\end{align}
where the last inequality is due to \eqref{eq1.27}.

On the other hand, it is clear that $v:=u-\tu^{\kappa r}-w$ satisfies
\[
\left\{
  \begin{aligned}
    L_n v = f_n-\tf_n^{\kappa r} \quad & \hbox{in $B_{\kappa r}$;} \\
    v=0 \quad & \hbox{on $\partial B_{\kappa r}$.}
  \end{aligned}
\right.
\]
Therefore, by the ABP maximum principle and Lemma~\ref{lem12.25} (ii) we have
\begin{multline}
                                                \label{eq2.14}
\abs{u-\tu^{\kappa r}-w}_{0; \,B_{\kappa r}}=\abs{v}_{0; \,B_{\kappa r}}\le N\kappa r \norm{f_n-\tf_n^{\kappa r}}_{L_d(B_{\kappa r})}\\
\le N(\kappa r)^{2+\delta} [f]_{x',\delta; \, B_{2\kappa r}}+ N \kappa r \norm{(a^{ij}_n-a^{ij})D_{ij}u}_{L_d(B_{2\kappa r})}.
\end{multline}
By Lemma \ref{lem12.25} (i), we also get
\begin{multline}
                                                \label{eq2.17}
\abs{\tu^{\kappa r}-\tT^2_{x_0'} \tu^{\kappa r}}_{0;\,B_r} \le Nr^3 \abs{D_{x'}^3 \tu^{\kappa r}}_{0; \,B_r} \le  N r^3 (\kappa r)^{\delta-1} [D^2_{x'} u]_{x', \delta;\, B_{(1+ \kappa) r}} \\
\le N\kappa^{\delta-1} r^{2+\delta}[u]_{x',2+\delta; \,B_{2\kappa r}}.
\end{multline}
Take $p=\tT^2_{x_0'}w+\tT^2_{x_0'} \tu^{\kappa r}\in  \tilde{\bP}_{2}$.
Then combining \eqref{Eq1.48} -- \eqref{eq2.17} yields
\begin{multline}			\label{eq3.09d}
\abs{u-p}_{0;\,B_r} \le \abs{u-\tu^{\kappa r}-w}_{0;\,B_r}+\abs{\tu^{\kappa r}-\tT^2_{x_0'} \tu^{\kappa r}}_{0;\,B_r}+\abs{w-\tT^2_{x_0'}w}_{0;\,B_r}\\
\le N\kappa^{\delta-1} r^{2+\delta}[u]_{x',2+\delta; \,B_{2\kappa r}}+N(\kappa r)^{2+\delta} [f]_{x',\delta;\, B_{2\kappa r}}+N\kappa r \norm{(a^{ij}_n-a^{ij})D_{ij}u}_{L_d(B_{2\kappa r})}.
\end{multline}
Letting $n\to \infty$, this together with the dominated convergence theorem implies
\begin{equation}
                                \label{eq2.28}
r^{-2-\delta}\,\inf_{p\in \tilde{\bP}_{2}} \abs{u-p}_{0;\,B_r(x_0)}\le N\kappa^{\delta-1} [u]_{x',2+\delta; \,B_{2\kappa r}(x_0)}+N \kappa^{2+\delta} [f]_{x',\delta; \,B_{2\kappa r}(x_0)},
\end{equation}
for any $x_0\in \bR^d$ and $r>0$. We take the supremum of the left-hand side \eqref{eq2.28} with respect to $x_0\in \bR^d$ and $r>0$, and then apply \cite[Theorem 3.3.1]{Kr96} to get
\[
[u]_{x',2+\delta}\le N\kappa^{\delta-1} [u]_{x',2+\delta}+N\kappa^{2+\delta} [f]_{x',\delta}.
\]
To finish the proof of \eqref{eq3.58} for $u\in C_{x'}^{2+\delta}(\bR^d)$, it suffices to choose a large $\kappa$ such that $N\kappa^{\delta-1}<1/2$.

Now we drop the assumption that $u\in C_{x'}^{2+\delta}(\bR^d)$ by another use of the partial mollification method.
As noted earlier in the proof, since $a^{ij}$ are independent of $x'$, we have
\[
L \tilde u^{1/n}=\tilde f^{1/n},\qquad n=1,2,\ldots.
\]
Since $\tilde u^{1/n}\in C_{x'}^{2+\delta}(\bR^d)$, by the argument above, we have a uniform estimate
\[
[\tilde u^{1/n}]_{x',2+\delta}\le N[\tilde f^{1/n}]_{x',\delta}\le N[f]_{x',\delta},\quad n=1,2,\ldots.
\]
Moreover, $[\tilde u^{1/n}]_0\le [u]_0$ and $\tilde u^{1/n}$ converges locally uniformly to $u$ as $n$ tends to infinity.
We thus conclude that  $u\in C_{x'}^{2+\delta}(\bR^d)$ and \eqref{eq3.58} holds. This completes the proof of assertion (i).

Next we prove assertion (ii).
In view of the proof above, without loss of generality we may assume that $a^{ij}$ are continuous in $B_1$ and $u\in C_{x'; \,loc}^{2+\delta}(B_1)$.
For $n=1,2,\ldots$, denote $r_n=3/4-2^{-n-1}$ and $B^{(n)}=B_{r_n}$. Note that $r_{n+1}-r_n=2^{-n-2}$ and $B^{(1)}=B_{1/2}$. Now we fix a point $x_0\in B^{(n)}$. Let $\kappa>2$ be a number to be fixed later.
For any $r\le 2^{-n-3}/\kappa$, we have $B_{2 \kappa r}(x_0)\subset B^{(n+1)}$.
It then follows from the proof of \eqref{eq2.28} in the previous step that
\begin{equation}
                                \label{eq2.41}
r^{-2-\delta}\,\inf_{p\in \tilde\bP_{2}} \abs{u-p}_{0; \,B_r(x_0)}\le N\kappa^{\delta-1} [u]_{x',2+\delta;\,B^{(n+1)}}+N \kappa^{2+\delta} [f]_{x',\delta;\,B_1}.
\end{equation}
On the other hand, for any $r\in ( 2^{-n-3}/\kappa,1/4)$, we have
\begin{equation}
                                \label{eq2.43}
r^{-2-\delta}\,\inf_{p\in \tilde\bP_{2}} \abs{u-p}_{0;\,B_r(x_0)}\le r^{-2-\delta} \abs{u}_{0;\,B_r(x_0)}\le (2^{n+3}\kappa)^{2+\delta} \abs{u}_{0;\,B_1}.
\end{equation}
Combining \eqref{eq2.41} and \eqref{eq2.43}, and  then applying \cite[Theorem 3.3.1]{Kr96}, we get
\begin{equation}
                                \label{eq2.48}
[u]_{x',2+\delta; \,B^{(n)}}\le N\kappa^{\delta-1} [u]_{x',2+\delta;\,B^{(n+1)}}+N\kappa^{2+\delta} [f]_{x',\delta;\,B_1}+N(2^{n+3}\kappa)^{2+\delta} \abs{u}_{0; \,B_1}.
\end{equation}
We choose a $\kappa$ sufficiently large such that $N\kappa^{\delta-1}\le 1/10$. By multiplying both sides of \eqref{eq2.48} by $10^{-n}$ and then summing over $n=1,2,\ldots$, we reach
\begin{align}
                                \label{eq2.54}
\sum_{n=1}^\infty 10^{-n} &[u]_{x',2+\delta;\,B^{(n)}}\nonumber\\
&\le \sum_{n=1}^\infty  10^{-n-1} [u]_{x',2+\delta;\,B^{(n+1)}}+N[f]_{x',\delta; \,B_1}+
N \sum_{n=1}^\infty 10^{-n} 2^{n(2+\delta)} \abs{u}_{0; \,B_1}\nonumber\\
&\le \sum_{n=1}^\infty  10^{-n-1} [u]_{x',2+\delta;\,B^{(n+1)}}+N[f]_{x',\delta;\, B_1}+N \abs{u}_{0; \,B_1}.
\end{align}
Since $u\in C^{2+\delta}_{x'}(B_{3/4})$, the summations in \eqref{eq2.54} are finite. By absorbing the first term on the right-hand side of \eqref{eq2.54} to the left-hand side, we get \eqref{eq3.58b}.

Finally, we prove assertion (iii) by combining the proof of \cite[Theorem 2.14]{DK11} with that of assertion (ii).
The theorem is proved.
\hfill\qedsymbol

\subsection{Proof of Theorem~\ref{thm4m}}
In order to prove the theorem, we need a slight generalization of the main result of \cite{Li86}, which can be proved in the same way as in \cite{Li86} by using dilations and standard approximation arguments.
\begin{lemma}
                                                  \label{lem4}
Let $r>0$ and $w\in   W^{2}_{d}(B_{r})$ be a function such that $w=0$ on $\partial B_r$.
Then there are constants $\epsilon \in (0,1]$ and $N$, depending only on $d$ and $\nu$, such that we have
\[
\fint_{B_{r}} \abs{D^2 w}^\epsilon \,dx   \leq N \left(\fint_{B_{r}} \abs{Lw}^{d} \,dx \right)^{\epsilon/ d}.
\]
\end{lemma}

Let $\delta_0=\delta_0(d,\nu)>0$ be the H\"older exponent appearing in the Krylov--Safonov estimate. We first assume that $D_{x'}^2u\in C^{\hd}(B_{1/2})$.
By mollification, we can find a sequence of coefficient matrices $a^{ij}_n$, which are continuous in $x$, $\delta$-H\"older continuous in $x'$ with $[a^{ij}_n]_{x',\delta}\le [a^{ij}]_{x',\delta}$, satisfy \eqref{elliptic}, and $a^{ij}_n\to a^{ij}$ a.e. as $n\to \infty$.
Let $L_n$ be the corresponding operator with $a^{ij}_n$ in place of $a^{ij}$.
Then we have
\[
L_n u=f_n,\quad \text{where}\quad f_n=f+(a^{ij}_n-a^{ij})D_{ij}u.
\]
We take a point $x_0\in B_{1/2}$ and $r,R\in (0,1/4)$ such that $0<r<R/4$.
Clearly, $u$ satisfies
\begin{equation}
                        \label{eq8.54}
a_n^{ij}(x_0',x'')D_{ij}u=f_n+g_n,
\end{equation}
where $g_n=\left(a_n^{ij}(x_0',x'')-a_n^{ij}\right)D_{ij}u$.
By the classical $W^2_d$ solvability for elliptic equations with continuous coefficients, there is a unique solution $w\in W^2_d(B_{R}(x_0))$ of the equation
\[
a_n^{ij}(x_0',x'')D_{ij}w=f_n-f(x_0',x'')+g_n
\]
in $B_{R}(x_0)$ with the zero Dirichlet boundary condition.
Thanks to Lemma \ref{lem4}, the triangle inequality, and H\"older's inequality, we have
\begin{align}
\int_{B_R(x_0)}  \Abs{D^2 w}^\epsilon \,dx
&\le NR^{d}\left(\fint_{B_R(x_0)} \Abs{f_n-f(x_0',x'')+g_n}^d \,dx\right)^{\epsilon/d}
\nonumber\\
&\le NR^{d-\epsilon}\left(\int_{B_R(x_0)} \abs{f_n-f}^d \,dx \right)^{\epsilon/d} \nonumber\\
                            \label{eq11.41m}
&\qquad+NR^{d+\epsilon\hd} [f]^\epsilon_{x',\hd; \,B_1}+N R^{d+\epsilon(\delta-d/p)} [a]_{x', \delta; \,B_1}^\epsilon \norm{D^2 u}_{L_p(B_1)}^\epsilon,
\end{align}
where $N=N(d, \nu)$.
It is easily seen that $v:=u-w\in  W^2_d(B_{R}(x_0))$ satisfies
\begin{equation}
                            \label{eq11.47m}
a^{ij}_n(x_0',x'')D_{ij}v=f(x_0',x'')\quad \text{in}\quad B_{R}(x_0).
\end{equation}
Note that both $a^{ij}(x_0',x'')$ and $f(x_0',x'')$ are independent of $x'$.
By mollification with respect to $x'$, without loss of generality, we may assume that $v$ is smooth with respect to $x'$.
By differentiating \eqref{eq11.47m} with respect to $x'$ twice, we see that $\hat v:=D_{x'}^2 v$ satisfies
\[
a_n^{ij}(x_0',x'')D_{ij}\hat v=0\quad \text{in}\quad B_{R/2}(x_0).
\]
Clearly, for any constant $c \in \bR$, the same equation is satisfied by $\tilde v:=\hat v-c$ in place of $\hat v$.
Denote $(\hat{v})_{B_r(x_0)}=\fint_{B_r(x_0)} \hat{v}$.
By applying the Krylov--Safonov estimate, we get
\begin{multline}
                                \label{eq12.31m}
\int_{B_r(x_0)} \Abs{\hat v-(\hat v)_{B_r(x_0)}}^\epsilon \,dx
=\int_{B_r(x_0)} \Abs{\tilde v-(\tilde v)_{B_r(x_0)}}^\epsilon \,dx
\le Nr^{d+\epsilon\delta_0}[\tilde v]^\epsilon_{\delta_0; \,B_{R/4}(x_0)} \\
\le N\left(\frac{r}{R}\right)^{d+\epsilon\delta_0}\int_{B_{R/2}(x_0)} \Abs{\tilde v}^\epsilon \,dx
=N\left(\frac{r}{R}\right)^{d+\epsilon\delta_0}\int_{B_{R/2}(x_0)} \Abs{\hat v-c}^\epsilon \,dx.
\end{multline}
Here, we recall the facts that for all $a, b \ge 0$ we have
\[
(a+b)^\epsilon \le a^\epsilon + b^\epsilon, \quad (a^\epsilon+ b^\epsilon) \le 2(a+b)^\epsilon.
\]
By \eqref{eq11.41m}, \eqref{eq12.31m}, and the above inequalities, we obtain \begin{align}
                                        \label{eq12.43m}
\int_{B_r(x_0)} & \Abs{D_{x'}^2u-(D_{x'}^2 v)_{B_r(x_0)}}^\epsilon \nonumber\\
&\le N\int_{B_r(x_0)} \Abs{D_{x'}^2v-(D_{x'}^2v)_{B_r(x_0)}}^\epsilon +N \int_{B_r(x_0)} \abs{D_{x'}^2 w}^\epsilon \nonumber\\
&\le N\left(\frac{r}{R}\right)^{d+\epsilon\delta_0}\int_{B_R(x_0)} \Abs{D_{x'}^2 v-c}^\epsilon  +N \int_{B_R(x_0)} \Abs{D^2 w}^\epsilon	\nonumber\\
&\le N \left(\frac{r}{R}\right)^{d+\epsilon\delta_0}\int_{B_R(x_0)} \Abs{D_{x'}^2 u-c}^\epsilon
+NR^{d-\epsilon} \left(\int_{B_R(x_0)} \abs{f_n-f}^d \right)^{\epsilon/d}	\nonumber\\
&\quad +NR^{d+\epsilon\hd} [f]^\epsilon_{x',\hd; \,B_1}+NR^{d+\epsilon(\delta-d/p)} [a]_{x', \delta; \,B_1}^\epsilon \norm{D^2 u}_{L_p(B_1)}^\epsilon.
\end{align}
Taking $n\to \infty$ in \eqref{eq12.43m}, by the dominated convergence theorem, we reach
\begin{multline*}
\int_{B_r(x_0)}  \Abs{D_{x'}^2 u-(D_{x'}^2 v)_{B_r(x_0)}}^\epsilon \le N\left(\frac{r}{R}\right)^{d+\epsilon\delta_0} \int_{B_R(x_0)} \Abs{D_{x'}^2 u-c}^\epsilon \\
+NR^{d+\epsilon\hd} [f]^\epsilon_{x',\hd; \,B_1}+NR^{d+\epsilon(\delta-d/p)} [a]_{x', \delta; \,B_1}^\epsilon \norm{D^2 u}_{L_p(B_1)}^\epsilon.
\end{multline*}
We set \footnote{By abuse of notation, we use $D^2_{x'}u$ to denote the scalar $D_{ij}u$ for $i,j=1,\ldots, q$.}
\[
\phi(x_0,r):=\inf_{c \in \bR} \int_{B(x_0,r)} \abs{D^2_{x'}u - c}^\epsilon.
\]
Note that since $c \in \bR$ is arbitrary, we get from the above inequality that
\begin{multline}		                                        \label{eq12.43m1}
\phi(x_0,r) \le N\left(\frac{r}{R}\right)^{d+\epsilon\delta_0} \phi(x_0, R)\\
+NR^{d+\epsilon\hd} [f]^\epsilon_{x',\hd; \,B_1}+NR^{d+\epsilon(\delta-d/p)} [a]_{x', \delta; \,B_1}^\epsilon \norm{D^2 u}_{L_p(B_1)}^\epsilon.
\end{multline}
The following lemma is a is variant of \cite[Lemma 2.1, p. 86]{Giaq83}, the main distinction from which is that the monotonicity of $\phi$ is not assumed below.
\begin{lemma}		\label{lem:giaq}
Let $\phi(t)$ be a nonnegative, bounded function on $(0,R_0]$
such that
\begin{equation}	\label{eq:giap86}
\phi (\rho) \le A \left[(\rho/R)^\alpha+\epsilon \right] \phi(R)+B R^\beta
\end{equation}
for all $0<\rho \le R \le R_0$, where $A$, $\alpha$, $\beta$ are nonnegative constants and $\beta<\alpha$.
Then there exists a constant $\epsilon_0=\epsilon_0(A,\alpha,\beta)$ such that if $\epsilon<\epsilon_0$, for all $0<\rho \le R_0$ we have
\[
\phi(\rho) \le C\left[ R_0^{-\beta} \left(\sup_{(0,R_0]} \phi \right) \rho^\beta +B \rho^\beta \right]
\]
with a constant  $C=C(A, \alpha, \beta)$.
\end{lemma}
\begin{proof}
Let $\psi(t)= \sup_{0<s\le t} \phi(s)$.
Then $\psi$ is a nonnegative, nondecreasing function on $(0,R_0]$.
For any $\tau \in (0, 1]$, we obtain from \eqref{eq:giap86} that
\[
\phi (\tau \rho) \le A \left[(\rho/R)^\alpha+\epsilon \right] \phi(\tau R)+B(\tau R)^\beta.
\]
Then by taking the supremum over $\tau \in (0,1]$, we find $\psi$ also satisfies the inequality \eqref{eq:giap86}.
Therefore, by \cite[Lemma 2.1, p. 86]{Giaq83}, for all $0<\rho \le R_0$ we have
\[
\psi (\rho) \le C \left[(\rho/R)^\beta \psi(R)+B \rho^\beta \right],
\]
where $C=C(A, \alpha, \beta)$.
\end{proof}

Observe that $\phi(x_0, r) \le \int_{B(x_0,r)}\, \abs{D^2_{x'}u}^\epsilon$ and thus we have
\[
\sup \set{\phi(x_0, r) : x_0 \in B_{1/2}, \,0<r<1/2}  \le N(\epsilon) \left(\int_{B_1} \abs{D^2 u} \,dx\right)^\epsilon.
\]
Therefore, by Lemma~\ref{lem:giaq}, we get from \eqref{eq12.43m1} that for any $x_0 \in B_{1/2}$ and any $r \in (0,1/16)$, there is a number  $c_{x_0,r}$ such that\footnote{The infimum in the definition of $\phi(x_0,r)$ is realized by some number $c$ in $\bR$.}
\begin{equation}
                                        \label{eq2.31m}
\int_{B(x_0,r)} \Abs{D^2_{x'}u - c_{x_0,r}}^\epsilon \le Nr^{d+\epsilon\hd}\left( [f]^\epsilon_{x',\hd; \,B_1} +\left(1+ [a]_{x', \delta;\, B_1}^\epsilon \right) \norm{D^2 u}^\epsilon_{L_p(B_1)}\right).
\end{equation}
By Campanato's characterization of H\"older continuous functions (see Lemma \ref{lemma6.14}), we obtain \eqref{eq3.59} from \eqref{eq2.31m}.

To remove the additional assumption $D_{x'}^2u\in C^{\hd}(B_{1/2})$, we use a mollification argument. Taking the mollification of \eqref{eq8.54} with respect to $x'$ and then differentiating with respect to $x'$ twice, we get that for any $\theta\in (0,1/4)$,
\[
a_n^{ij}(x_0',x'')D_{ij}\left(D_{x'}^2\tilde u^\theta\right)=D_{x'}^2\tilde f_n^\theta+D_{x'}^2\tilde g_n^\theta\quad \text{in}\quad B_{3/4}.
\]
Thanks to the Krylov--Safonov estimate, $D_{x'}^2\tilde u^\theta\in C^{\delta_0}(B_{1/2})$.
Then by the proof above with $u^\theta$, $\tilde f_n^\theta$, $\tilde f^\theta$, and $\tilde g_n^\theta$ in place of $u$, $f_n$, $f$, and $g_n$, respectively,  we get \eqref{eq2.31m} with $u$ replaced by $\tilde u^\theta$ on the left-hand side. Therefore, by Lemma \ref{lemma6.14},
\[
[D_{x'}^2\tilde u^\theta]_{\hd; \,B_{1/2}}\le
N\left( [f]_{x',\hd; \,B_1} +\left(1+ [a]_{x', \delta; \,B_1}\right) \norm{D^2 u}_{L_p(B_1)}\right),
\]
where $N$ is independent of $\theta$.
Since $D_{x'}^2\tilde u^\theta\to D_{x'}^2 u$ a.e. as $\theta\to 0$ by the Lebesgue lemma, we obtain \eqref{eq3.59}.
This completes the proof of assertion (i).

The proof of assertion (ii) is similar.
We provide the details for the completeness.
Take a point $x_0\in B_{1/2}$ and $r,R\in (0,1/4)$ such that $0<r<R/4$.
Let $w\in W^1_2(B_{R}(x_0))$ be the weak solution of the equation
\[
D_i \left(a^{ij}(x_0',x'')D_{j}w \right)=\dv \left(\vec f(x)-\vec f(x_0',x'')\right) +D_i \left((a^{ij}(x_0',x'')-a^{ij}(x))D_{j}u\right)
\]
in $B_{R}(x_0)$ with the zero Dirichlet boundary condition.
By the $W^1_2$ estimate and H\"older's inequality, we have
\begin{align}
\int_{B_R(x_0)} \abs{Dw}^2 \,dx
&\le N\int_{B_R(x_0)} \Abs{ \vec f(x)-\vec f(x_0',x'') + \left(a^{ij}(x_0',x'')-a^{ij}(x)\right) D_{j}u}^2 \,dx\nonumber\\
                            \label{eq11.41md}
&\le N \left(R^{d+2\hd} [\vec f]^2_{x',\hd; \,B_1}+R^{d+2(\delta-d/p)} [a]_{x', \delta; \,B_1}^2 \norm{Du}_{L_{p}(B_1)}^2\right),
\end{align}
where $N=N(d, \nu)$.
It is easily seen that $v:=u-w\in  W^1_2(B_{R}(x_0))$ satisfies
\begin{equation}
                            \label{eq11.47md}
D_i \left(a^{ij}(x_0',x'')D_{j}v \right)=\dv \vec f(x_0',x'')\quad \text{in}\quad B_{R}(x_0).
\end{equation}
Note that both $a^{ij}(x_0',x'')$ and $\vec f(x_0',x'')$ are independent of $x'$.
By mollification with respect to $x'$, without loss of generality, we may assume that $v$ is smooth with respect to $x'$.
By differentiating \eqref{eq11.47md} with respect to $x'$, we see that $\hat v:=D_{x'} v$ satisfies
\begin{equation}
				\label{eq3.38ca}
D_i \left(a^{ij}(x_0',x'')D_{j}\hat v\right)=0\quad \text{in}\quad B_{R/2}(x_0).
\end{equation}
Clearly, the above equation is still satisfied by $\tilde v:=\hat v-(\hat v)_{B_R(x_0)}$ in place of $\hat v$.
Therefore, by applying the De Giorgi--Nash--Moser estimate, we get
\begin{align}
                                \label{eq12.31md}
\int_{B_r(x_0)} & \Abs{\hat v-(\hat v)_{B_r(x_0)}}^2 \,dx
=\int_{B_r(x_0)}\Abs{\tilde v-(\tilde v)_{B_r(x_0)}}^2 \,dx
\le Nr^{d+2\delta_0} [\tilde v]^2_{\delta_0; \,B_{R/4}(x_0)}\nonumber\\
&\le N\left(\frac{r}{R}\right)^{d+2\delta_0}\int_{B_{R/2}(x_0)} \Abs{\tilde v}^2 \,dx
= N\left(\frac{r}{R}\right)^{d+2\delta_0}\int_{B_{R/2}(x_0)} \Abs{\hat v-(\hat v)_{B_R(x_0)}}^2 \,dx,
\end{align}
where  $\delta_0=\delta_0(n, \nu)>0$ is the H\"older exponent appearing in the De Giorgi--Nash--Moser estimate.
By \eqref{eq11.41md}, \eqref{eq12.31md}, and the triangle inequality, we obtain
\begin{align}
                                        \label{eq12.43md}
\int_{B_r(x_0)} & \Abs{D_{x'}u-(D_{x'}u)_{B_r(x_0)}}^2 \,dx\nonumber\\
&\le N\int_{B_r(x_0)}\Abs{D_{x'}v-(D_{x'}v)_{B_r(x_0)}}^2 \,dx+N
\int_{B_r(x_0)}\Abs{D_{x'}w}^2 \,dx	\nonumber\\
&\le N\left(\frac{r}{R}\right)^{d+2\delta_0}\int_{B_{R/2}(x_0)} \Abs{D_{x'}v-(D_{x'} v)_{B_{R}(x_0)} }^2 \,dx +N \int_{B_R(x_0)}\Abs{D w}^2 \,dx
\nonumber\\
&\le N\left(\frac{r}{R}\right)^{d+2\delta_0}\int_{B_{R/2}(x_0)} \Abs{D_{x'} u-(D_{x'} u)_{B_{R}(x_0)}}^2\,dx + N \int_{B_R(x_0)}\Abs{D w}^2 \,dx	\nonumber\\
&\le N\left(\frac{r}{R}\right)^{d+2\delta_0}\int_{B_{R}(x_0)} \Abs{D_{x'} u-(D_{x'} u)_{B_{R}(x_0)}}^2\,dx	\nonumber\\
&\omit\hfill $+NR^{d+2\hd} \left([\vec f]^2_{x',\hd; \,B_1}
+ [a]_{x', \delta; \,B_1}^2 \norm{Du}^2_{L_p(B_1)}\right)$.
\end{align}
By Lemma~\ref{lem:giaq}, we infer from \eqref{eq12.43md} that, for all $0<r<1/16$, we have
\begin{align}
                                        \label{eq2.31md}
\int_{B_r(x_0)} &\Abs{D_{x'}u-(D_{x'}u)_{B_r(x_0)}}^2 \,dx	\nonumber\\
&\le N r^{d+2\hd}\left(\int_{B_{1/4}(x_0)} \Abs{D_{x'} u}^2 \,dx+[\vec f]^2_{x',\hd; \,B_1}
+ [a]_{x', \delta; \,B_1}^2 \norm{Du}^2_{L_p(B_1)}\right)	\nonumber\\
&\le Nr^{d+2\hd}\left([\vec f]^2_{x',\hd; \,B_1} + \left(1+ [a]_{x', \delta; \,B_1}^2\right) \norm{Du}^2_{L_p(B_1)}\right).
\end{align}
Then we get \eqref{eq4.00} from \eqref{eq2.31md}.
The theorem is proved.
\hfill\qedsymbol

\subsection{Proof of Theorem~\ref{thm4mn}}
For the simplicity of presentation, in Lemma \ref{lem:sss} (ii) and the proof of Theorem~\ref{thm4mn} below, we slightly abuse the notation by still using $u$, $a^{ij}$, and $\cL$ to denote $(u_1,\ldots,u_m)^T$, $[a^{ij}_{\alpha\beta}]_{\alpha,\beta=1}^m$, and $D_i (a^{ij}_{\alpha\beta}D_j)$, respectively, where $m$ is a positive integer; that is, we keep using the scalar notation for the systems as well.
The following lemma is a key for the proof.
\begin{lemma}		\label{lem:sss}
Assume that $a^{ij}(x)=a^{ij}(x'')=a^{ij}(x^{d-1},x^d)$.
There exist constants $\delta_0=\delta_0(d, \nu) \in (0,1)$ and $N=N(d,\nu)$ such that the following hold.
\begin{enumerate}[(i)]
\item
If $u \in W^2_2(B_1)$ is a strong solution of $Lu=0$ in $B_1$, then we have
\[
[Du]_{\delta_0; \,B_{1/2}} \le N \norm{u}_{L_2(B_1)}.
\]
\item
If $u \in W^1_2(B_1)$ is a weak solution of a strongly elliptic system $\cL u =0$ in $B_1$, then we have
\[
[u]_{\delta_0; \,B_{1/2}} \le N \norm{u}_{L_2(B_1)}.
\]
\end{enumerate}
\end{lemma}
\begin{proof}
In the proof we shall denote $I_r=(-r,r)^d$, $I_r'=(-r,r)^{d-2}$, and $I_r''=(-r,r)^2$.
We begin with assertion (i).
We recall that there exist constants $p_0=p_0(d, \nu)>2$ and $N=N(d,\nu)$ such that  if $u \in W^{2}_2(I_2)$ is a strong solution of $Lu=0$ in $I_2$, then $u \in W_{p_0}^2(I_1)$ and
\begin{equation}		\label{eq:dokr}
\norm{u}_{W^2_{p_0}(I_1)} \le N \norm{u}_{L_{p_0}(I_{2})}.
\end{equation}
Indeed, for $1 \le s < t \le 2$ fixed, let $\hat\eta$ be a smooth function on $\bR$ with a compact support in $(-s,s)$ such that $\hat\eta \equiv 1$ on $(-t,t)$, $\abs{\hat\eta'}_0 \le 2/(s-t)$, and $\abs{\hat\eta''}_0 \le 6/(s-t)^2$.
By \cite[Theorem~2.8]{DoKr10} applied to $v=\hat\zeta u$, where $\hat\zeta(x)=\prod_{i=1}^d \hat\eta(x^i)$, followed by an interpolation inequality for the Sobolev spaces, we have
\begin{align*}
\norm{D^2 u}_{L_{p_0}(I_t)} &\le N (s-t)^{-2}\, \norm{u}_{L_{p_0}(I_s)} + N (s-t)^{-1}\, \norm{Du}_{L_{p_0}(I_s)} \\
& \le N \epsilon^{-1} (s-t)^{-2}\, \norm{u}_{L_{p_0}(I_s)} + \epsilon \norm{D^2u}_{L_{p_0}(I_s)}.
\end{align*}
Then, by \cite[Lemma~3.1, p. 161]{Giaq83}, we obtain
\[
\norm{D^2 u}_{L_{p_0}(I_1)} \le N \norm{u}_{L_{p_0}(I_{2})}.
\]
This inequality and interpolation inequalities for the Sobolev spaces, we get \eqref{eq:dokr}.
Since we assume that $a^{ij}$ are independent of $x'$, by using the argument of finite-difference quotients, the same inequality also holds
for $D_{x'}^j u$ ($j=1, 2, \ldots$ ) in place of $u$.
Then by successive application of \eqref{eq:dokr} together with standard covering argument, we obtain
\begin{equation}		\label{eq2rvh}
\norm{D^2 D^j_{x'} u}_{L_{p_0}(I_1)} + \norm{D D^j_{x'} u}_{L_{p_0}(I_1)} \le N \norm{u}_{L_{p_0}(I_{3/2})}\le N \norm{u}_{L_2(I_2)},\quad j=0,1,2,\ldots,
\end{equation}
where we used Krylov--Safonov estimate in the last inequality; that is,
\[
\norm{u}_{L_\infty(I_{3/2})}  \le N \norm{u}_{L_2(I_2)}.
\]
Take $\delta_0= 1-2/p_0$ and set $v=Du$.
By the Morrey-Sobolev inequality, we have
\begin{equation}		\label{eq:sevit0}
\sup_{\substack{x'', y'' \in I_1'' \\ x'' \neq y''}}\frac{\abs{v(x', x'')-v(x', y'')}}{\abs{x''-y''}^{\delta_0}} \le N \left( \norm{D v(x',\cdot\,)}_{L_{p_0}(I_1'')}+\norm{v(x',\cdot\,)}_{L_{p_0}(I_1'')}\right).
\end{equation}
On the other hand, by the Sobolev embedding, there exists $k \in \bN$ such that $v(x',x'')$ and $D_{x''}v(x',x'')$ as functions of $x' \in I_1'$ satisfy
\[
\sup_{x' \in I_1'} \left( \abs{v(x', x'')} + \abs{D_{x''}v(x', x'')}\right) \le N  \left(\norm{v(\,\cdot\,, x'')}_{W^k_{p_0}(I_1')}+\norm{D_{x''}v(\,\cdot\,, x'')}_{W^k_{p_0}(I_1')}\right).
\]
This implies that for all $x'\in I_1'$ we have
\[
\int_{I_1''} \abs{v(x',x'')}^{p_0} + \abs{D_{x''}v(x',x'')}^{p_0} \,dx'' \le N  \sum_{0 \le j \le k}\left(\norm{D_{x'}^j v}_{L_p(I_1)}^{p_0}+\norm{D_{x'}^j D_{x''} v}_{L_p(I_1)}^{p_0}\right).
\]
This combined with \eqref{eq:sevit0} and \eqref{eq2rvh} shows that (recall $v=Du$)
\begin{equation}		\label{eq:doong0}
[Du]_{x'', \delta_0; \,I_1} \le N  \norm{u}_{L_2(I_2)}.
\end{equation}
Next, again by the Sobolev embedding theorem, we find a positive integer $k$ such that $v(x',x'')$, as a function of $x' \in I_1'$, satisfies
\begin{equation}		\label{eq:sevit00}
\sup_{\substack{x', y' \in I_1' \\ x' \neq y'}}\frac{\abs{v(x', x'')-v(y', x'')}}{\abs{x'-y'}^{\delta_0}} \le N \norm{v(\,\cdot\,,x'')}_{W^k_{p_0}(I_1')}.
\end{equation}
By the Morrey-Sobolev inequality, for $j=0, 1, \ldots, k$, $D^j_{x'} v(x', x'')$, as a function of $x'' \in I_1''$, satisfies
\[
\sup_{x'' \in I_1''}\, \abs{D^j_{x'} v(x', x'')} \le N \left( \norm{D^j_{x'} v(x', \cdot\,)}_{L_{p_0}(I_1'')} + \norm{D_{x''}D^j_{x'} v(x', \cdot\,)}_{L_{p_0}(I_1'')}\right).
\]
This together with \eqref{eq:sevit00} and \eqref{eq2rvh} gives (recall $v=Du$)
\begin{equation}		\label{eq:doong00}
[Du]_{x', \delta_0;\, I_1} \le N \sum_{0\le j \le k} \left(\norm{D_{x'}^j Du}_{L_{p_0}(I_1)} + \norm{D_{x''}D_{x'}^j Du}_{L_{p_0}(I_1)} \right) \le N \norm{u}_{L_2(I_2)}.
\end{equation}
By combining \eqref{eq:doong0} and \eqref{eq:doong00} and using a standard covering argument, we get assertion (i).

The proof of assertion (ii) is similar.
Recall that there exist constants $p_0=p_0(d,\nu)>2$ and $N=N(d,\nu)$ such that if $u \in W^1_2(I_2)$ is a weak solution of $\cL u =0$ in $I_2$, then $Du$ is $p_0$-integrable in $I_{1}$ and
\[
\norm{Du}_{L_{p_0}(I_{1})} \le N \norm{Du}_{L_2(I_{3/2})}.
\]
Also, note that by the Sobolev embedding theorem, we have
\[
\norm{u}_{L_{2d/(d-2)}(I_{1})} \le N \norm{u}_{W^1_2(I_1)}.
\]
Therefore, by the Caccioppoli inequality, we get (by replacing $p_0$ by $2d/(d-2)$ and using H\"older's inequality if necessary)
\[
\norm{u}_{L_{p_0}(I_{1})}+ \norm{Du}_{L_{p_0}(I_{1})} \le N \norm{u}_{L_2(I_2)}.
\]
Since we assume $a^{ij}$ are independent of $x'$, the same inequality also holds
for $D_{x'}^j u$ ($j=1, 2, \ldots$ ) in place of $u$.
Then by successive application of the Caccioppoli inequality together with a standard covering argument, we obtain
\begin{equation}		\label{eq3rvh}
\norm{D^j_{x'}u}_{L_{p_0}(I_{1})} + \norm{DD^j_{x'}u}_{L_{p_0}(I_{1})} \le N(d, \nu, j) \,\norm{u}_{L_2(I_2)},\quad j=0,1,2,\ldots.
\end{equation}
Then by using \eqref{eq3rvh} instead of \eqref{eq2rvh} and repeating the same argument \eqref{eq:sevit0}--\eqref{eq:doong00} with $u$ in place of $v$, we get
assertion (ii).
The lemma is proved.
\end{proof}

We begin with the proof of assertion (i).
We follow exactly the same line of the proof of Theorem~\ref{thm4m} (i) up to \eqref{eq11.47m}.
As before, by the mollification argument we may assume that $DD_{x'}u\in C^{\hd}(B_{1/2})$.
Recall that $x''=(x^{d-1}, x^d)$ and observe that $\hat{v}:= D_{x'} v$ satisfies the equation
\begin{equation}		\label{eq:pout}
a^{ij}_n (x_0', x^{d-1}, x^d) D_{ij} \hat{v}=0\quad \text{in}\quad B_{R/2}(x_0).
\end{equation}
By Lemma~\ref{lem:sss} with a scaling, we find that
\begin{equation}		\label{eq:213b}
[D \hat{v}]_{\delta_0; \,B_{R/4}(x_0)} \le  N R^{-d/2-1-\delta_0} \norm{\hat{v}}_{L_2(B_{3R/8}(x_0))} \le N R^{-d/\epsilon-1-\delta_0} \norm{\hat{v}}_{L_\epsilon(B_{R/2}(x_0))},
\end{equation}
where we used Krylov--Safonov estimate in the last inequality; that is,
\[
\norm{\hat v}_{L_\infty(B_{3R/8}(x_0))}  \le N
\left(\fint_{B_{R/2}(x_0)} \abs{\hat v}^\epsilon \right)^{1/\epsilon}.
\]
Note that \eqref{eq:pout} is still satisfied with
\[
\tilde v:=\hat v-(\hat v)_{B_{R/2}(x_0)}-\vec c \cdot (x-x_0)
\]
in place of $\hat v$, where $\vec c \in \bR^d$ is an arbitrary constant vector.
It should be noted that $(\tilde{v})_{B_{R/2}(x_0)}=0$.
Therefore, by applying \eqref{eq:213b} and the Poincar\'e inequality, we get
\begin{align}
                                \label{eq12.31mnp}
\int_{B_r(x_0)} & \Abs{D\hat v-(D\hat v)_{B_r(x_0)}}^\epsilon \,dx
=\int_{B_r(x_0)} \Abs{D\tilde v-(D\tilde v)_{B_r(x_0)}}^\epsilon \,dx \nonumber \\
&\le Nr^{d+\epsilon \delta_0}[D\tilde v]^\epsilon_{\delta_0;\, B_{R/4}(x_0)}	\le Nr^{d+\epsilon\delta_0}R^{-d-\epsilon-\epsilon\delta_0}\int_{B_{R/2}(x_0)} \Abs{\tilde v}^\epsilon \,dx	\nonumber\\
& \le N\left(\frac{r}{R}\right)^{d+\epsilon\delta_0}\int_{B_{R/2}(x_0)} \Abs{D\tilde v}^\epsilon \,dx\nonumber\\
&= N \left(\frac{r}{R}\right)^{d+\epsilon\delta_0}\int_{B_{R/2}(x_0)} \Abs{D\hat v-\vec c}^\epsilon \,dx.
\end{align}
Then by using \eqref{eq12.31mnp} instead of \eqref{eq12.31m} and proceed as in the proof of Theorem~\ref{thm4m}, we find, similar to \eqref{eq2.31m}, that for any $x_0 \in B_{1/2}$ and any $r \in (0,1/16)$, there is a constant vector $\vec c_{x_0,r} \in \bR^d$ such that
\[
\int_{B_r(x_0)}  \Abs{DD_{x'}u-\vec c_{x_0 ,r}}^\epsilon \,dx \\
\le Nr^{d+\epsilon\hd} \left( [f]^\epsilon_{x',\hd; \,B_1}+ \left(1+[a]_{x', \delta; \,B_1}^\epsilon\right) \norm{D^2 u}_{L_p(B_1)}^\epsilon \right),
\]
from which \eqref{eq3.59mn} follows as before.
This completes the proof of assertion (i).

We now turn to assertion (ii), the proof of which is slightly different from that of Theorem~\ref{thm4m} (ii) in a way that we dispense with the De Giorgi--Nash--Moser estimate so that the proof carries over to the case of strongly elliptic systems.
First, we follow exactly the same proof of Theorem~\ref{thm4m} (ii) up to \eqref{eq3.38ca}.
Recall that we assumed that $x''=(x^{d-1}, x^d)$ and observe that \eqref{eq3.38ca} is still satisfied by $\tilde{v}:=\hat{v}-(\hat{v})_{B_R(x_0)}$.
Therefore, by Lemma~\ref{lem:sss} with a scaling (note that Lemma~\ref{lem:sss} (ii) holds for systems), we find that $\tilde{v}$ satisfies an estimate
\[
[\tilde{v}]_{\delta_0; \,B_{R/4}(x_0)} \le N R^{-\delta_0-d/2} \norm{\tilde{v}}_{L_2(B_{R/2}(x_0))}
\]
with $\delta_0=\delta_0(d, \nu)>0$ and $N=N(d, \nu)$.

Then, by utilizing the above instead of De Giorgi--Nash--Moser  estimate, we repeat the rest of proof of Theorem~\ref{thm4m} (ii) and obtain \eqref{eq4.00}.
The theorem is proved.
\hfill\qedsymbol

\subsection{Proof of Theorem~\ref{thm4}}
For the proof of the theorem we need the following three lemmas.
\begin{lemma}
                            \label{lem1}
Let $p\in (1,\infty)$ be a constant.
Assume that $a=[a^{ij}]$ are continuous with respect to $(x^1,\ldots,x^{d-1})$ in $\bar B_{2/3}$ with a modulus of continuity $\omega_a$ and continuous in $\bar B_{1}\setminus B_{2/3}$ with a modulus of continuity $\tilde \omega_a$.
Then for any $f\in L_p(B_1)$, there is a unique strong solution $u\in W^2_p(B_1)$ to the equation $a^{ij}D_{ij}u=f$ in $B_1$ with the Dirichlet boundary condition $u=0$ on $\partial B_1$.
Moreover, we have
\[
\norm{u}_{W^2_p(B_1)}\le N \norm{f}_{L_p(B_1)},
\]
where $N$ depends only on $d$, $p$, $\nu$, $\omega_a$, and $\tilde \omega_a$.
\end{lemma}
\begin{proof}
By following the arguments in \cite[Chapter 11]{Kr08}, the lemma is a consequence of the a priori interior $W^2_p$ estimates proved in \cite{KK07, Dong12b} for elliptic equations with coefficients measurable in one direction and the classical boundary $W^2_p$ estimate for elliptic equations with continuous coefficients.
\end{proof}

\begin{lemma}
                                \label{lem2}
Assume that $a=[a^{ij}]$ are continuous with respect to $(x^1,\ldots,x^{d-1})$ with a modulus of continuity $\omega_a$.
\begin{enumerate}[(i)]
\item
Let $p\in (1,\infty)$.
Assume that  $u\in W^2_p(B_1)$ and satisfies $a^{ij}D_{ij}u=f$ in $B_1$, where $f\in L_p(B_1)$.
Then we have
\[
\norm{u}_{W^2_p(B_{1/2})}\le N \left(\norm{f}_{L_p(B_1)}+\norm{u}_{L_p(B_1)}\right),
\]
where $N$ depends only on $d$,  $\nu$, $p$, and $\omega_a$.

\item
If in addition $f\in L_{\hp}(B_1)$ for some $\hp \in (p,\infty)$, then we have $u\in W^2_{\hp}(B_{1/2})$ and
\[
\norm{u}_{W^2_{\hp} (B_{1/2})}\le N \left( \norm{f}_{L_{\hp}(B_1)}+\norm{u}_{L_p(B_1)} \right),
\]
where $N$ depends only on $d$, $\nu$, $p$, and $\omega_a$.
In particular, if $\hp>d$, it holds that
\[
[Du]_{\gamma; \,B_{1/2}}\le N \left( \norm{f}_{L_{\hp}(B_1)}+\norm{u}_{L_p(B_1)} \right),
\]
where $\gamma=1-d/\hp$.
\end{enumerate}
\end{lemma}

\begin{proof}
The first assertion follows from the main result of \cite{Dong12b} by a standard localization argument.
The second assertion is a consequence of the first one, the Sobolev embedding theorem, and a bootstrap argument to successively improve the integrability.
\end{proof}

Here is a counterpart of Lemma~\ref{lem2} for divergence form equations.

\begin{lemma}
                                \label{lem3}
Assume that $a=[a^{ij}]$ are continuous with respect to $(x^1,\ldots,x^{d-1})$ with a modulus of continuity $\omega_a$.
\begin{enumerate}[(i)]
\item
Let $p\in (1,\infty)$.
Assume that $u\in W^1_p(B_1)$ and satisfies
$D_i(a^{ij}D_{j}u)=\dv \vec f$ in $B_1$, where $\vec f= (f^1,\ldots,f^d) \in L_{p}(B_1)$.
Then we have
\[
\norm{u}_{W^1_p(B_{1/2})}\le N \left( \norm{\vec f}_{L_p(B_1)}+\norm{u}_{L_p(B_1)} \right),
\]
where $N$ depends only on $d$,  $\nu$, $p$, and $\omega_a$.

\item
If in addition $\vec f\in L_{\hp}(B_1)$ for some ${\hp}\in (p,\infty)$, then we have $u\in W^1_{\hp}(B_{1/2})$ and
\[
\norm{u}_{W^1_{\hp}(B_{1/2})}\le N \left( \norm{\vec f}_{L_{\hp}(B_1)}+\norm{u}_{L_p(B_1)}\right),
\]
where $N$ depends only on $d$, $\nu$, $p$, and $\omega_a$.
In particular, if ${\hp}>d$, it holds that
\[
[u]_{\gamma; B_{1/2}}\le N \left( \norm{\vec f}_{L_{\hp}(B_1)}+\norm{u}_{L_p(B_1)}\right),
\]
where $\gamma=1-d/{\hp}$.
\end{enumerate}
\end{lemma}

\begin{proof}
The proof is similar to that of Lemma~\ref{lem2} by appealing to \cite[Theorem 2.2]{DK09}.
\end{proof}

Now, we turn to the proof of the theorem. We begin with assertion (i).
We take a point $x_0\in B_{1/2}$ and $r,R\in (0,1/4)$ such that $0<r<R/8$.
Let $\hat\zeta\in C_0^\infty(B_1)$ be a smooth cut-off function such that $0\le \hat\zeta\le 1$ in $B_1$, $\hat\zeta=1$ in $B_{1/2}$, and $\hat\zeta=0$ in $B_{1}\setminus B_{2/3}$.
Define
\[
\hat a^{ij}(x)=\hat\zeta ((x-x_0)/R )\,a^{ij}(x_0',x'') + \left(1-\hat\zeta((x-x_0)/R)\right)\delta_{ij}.
\]
Observe that $\hat a^{ij}$ are continuous with respect to $(x^1,\ldots,x^{d-1})$ in $\bar B_{2R/3}(x_0)$ and continuous (with respect to $x$) in $\bar B_{R}(x_0)\setminus B_{2R/3}(x_0)$.
By Lemma~\ref{lem1}, there is a unique solution $w\in W^2_p(B_{R}(x_0))$ of the equation
\[
\hat a^{ij}D_{ij}w=f(x)-f(x_0',x'') +(a^{ij}(x_0',x'')-a^{ij}(x))D_{ij}u
\]
in $B_{R}(x_0)$ with the zero Dirichlet boundary data.
It is crucial to observe that the modulus of continuity of $\hat{a}=\hat a^{ij}$ only improves under the affine transformation of $B_R(x_0)$ to $B_1$ as $R<1$.
Therefore, we have an estimate
\begin{align}
\norm{D^2 w}_{L_p(B_R(x_0))}&\le N \Norm{f(x)-f(x_0',x'') +(a^{ij}(x_0',x'')-a^{ij}(x))D_{ij}u}_{L_p(B_R(x_0))}\nonumber\\
                            \label{eq11.41}
&\le NR^{\hd+d/p}[f]_{x',\hd; \,B_1}
+NR^{\delta}  [a]_{x', \delta; \,B_1}\, \norm{D^2 u}_{L_p(B_R(x_0))}
\end{align}
with a constant $N=N(d, p, \nu, \omega_a)$ that is independent of $R \in (0, 1/4)$.

Since $\hat a^{ij}(x) =a^{ij}(x_0',x'')$ in $B_{R/2}(x_0)$, it is easily seen that $v:=u-w\in  W^2_p(B_{R}(x_0))$ satisfies
\begin{equation}
                            \label{eq11.47}
a^{ij}(x_0',x'')D_{ij}v=f(x_0',x'')\quad \text{in}\quad B_{R/2}(x_0).
\end{equation}
Note that both $a^{ij}(x_0',x'')$ and $f(x_0',x'')$ are independent of $x'$.
By mollification with respect to $x'$, without loss of generality, we may assume that
$\hat v:=D_{x'}v\in W^2_{p}(B_{R/4}(x_0))$.
By differentiating \eqref{eq11.47} with respect to $x'$, we see that $\hat v$ satisfies
\[
a^{ij}(x_0',x'')D_{ij}\hat v=0\quad \text{in}\quad B_{R/4}(x_0).
\]
Clearly, the equation above still holds with
\[
\tilde v:=\hat v-(\hat v)_{B_{R/4}(x_0)}-(x^i-x_0^i)(D_i \hat v)_{B_{R/4}(x_0)}
\]
in place of $\hat v$.

Take any $\gamma\in (\hd,1)$.
By applying Lemma \ref{lem2} with a scaling (the modulus of continuity only improves!) and the Poincar\'e inequality, we get (cf. \eqref{eq12.31mnp} above)
\begin{equation}
                                \label{eq12.31}
\int_{B_r(x_0)}  \Abs{D\hat v-(D\hat v)_{B_r(x_0)}}^p \,dx
\le N\left(\frac{r}{R}\right)^{d+p\gamma}\int_{B_{R/4}(x_0)} \Abs{D\hat v-(D\hat v)_{B_{R/4}(x_0)}}^p \,dx.
\end{equation}
By \eqref{eq11.41} and \eqref{eq12.31},  we reach (cf. \eqref{eq12.43md} above)
\begin{align}
                                        \label{eq12.43}
\int_{B_r(x_0)} &\Abs{DD_{x'}u-(DD_{x'}u)_{B_r(x_0)}}^p \,dx\nonumber\\
&\quad\le N\int_{B_r(x_0)}\Abs{DD_{x'}v-(DD_{x'}v)_{B_r(x_0)}}^p \,dx+N
\int_{B_r(x_0)}\Abs{DD_{x'}w}^p \,dx\nonumber\\
&\quad\le N\left(\frac{r}{R}\right)^{d+p\gamma}\int_{B_{R/4}(x_0)}\Abs{DD_{x'} u-(DD_{x'} u)_{B_{R/4}(x_0)}}^p \,dx \nonumber\\
&\omit\hfill $+NR^{d+p\hd}\left( [f]^p_{x',\hd; \,B_1} + [a]_{x', \delta; \,B_1}^p   \norm{D^2u}^p_{L_p(B_1)}\right).$
\end{align}
By Lemma~\ref{lem:giaq}, we infer from \eqref{eq12.43} that, for all $0<r<1/32$ we have
\begin{align*}
\int_{B_r(x_0)} &\Abs{DD_{x'}u-(DD_{x'}u)_{B_r(x_0)}}^p \,dx\nonumber\\
&\le Nr^{d+p\hd} \left(\int_{B_{1/16}(x_0)}\Abs{DD_{x'} u}^p \,dx+[f]^p_{x',\hd; \,B_1}
+ [a]_{x', \delta; \,B_1}^p \norm{D^2u}^p_{L_p(B_1)}\right)\nonumber\\
&\le Nr^{d+p\hd}\left([f]^p_{x',\hd; \,B_1}
+ \left(1+[a]_{x', \delta; \,B_1}^p\right) \norm{D^2u}^p_{L_p(B_1)}\right).
\end{align*}
Therefore, we obtain \eqref{eq4.11} by Campanato's theorem.
This completes the proof of assertion (i).

The proof of assertion (ii) is similar, and actually simpler. We give the details for the completeness.
We take a point $x_0\in B_{1/2}$ and $r,R\in (0,1/4)$ such that $0<r<R/4$. Let $w\in W^1_2(B_{R}(x_0))$ be the weak solution of the equation
\[
D_i \left(a^{ij}(x_0',x'')D_{j}w\right)=\dv \left(\vec f(x)-\vec f(x_0',x'')\right)+D_i \left(\left( a^{ij}(x_0',x'')-a^{ij}(x)\right)D_{j}u\right)
\]
in $B_{R}(x_0)$ with the Dirichlet boundary condition $w=0$ on $\partial B_{R}(x_0)$.
By the classical $L_2$ estimate for divergence form equations and H\"older's inequality, we have
\begin{align}
\norm{Dw}_{L_2(B_R(x_0))}&\le N \Norm{\vec f(x)-\vec f(x_0',x'') + (a^{ij}(x_0',x'')-a^{ij}(x) )D_{j}u}_{L_2(B_R(x_0))}\nonumber\\
                            \label{eq11.41d}
&\le NR^{\hd+d/2} [\vec f]_{x',\hd; \,B_1}
+NR^{\hd+d/2} [a]_{x', \delta; \,B_1} \norm{Du}_{L_p(B_R(x_0))},
\end{align}
where $N$ depends only on $d$, $p$, and $\nu$. It is easily seen that $v:=u-w\in  W^1_2(B_{R}(x_0))$ satisfies
\begin{equation}
                            \label{eq11.47d}
D_i \left(a^{ij}(x_0',x'')D_{j}v \right)=\dv \vec f(x_0',x'')\quad \text{in}\quad B_{R}(x_0).
\end{equation}
By mollification with respect to $x'$, without loss of generality, we may assume that $\hat v:=D_{x'}v\in W^1_{p}(B_{R/2}(x_0))$. By differentiating \eqref{eq11.47d} with respect to $x'$, we see that $\hat v$ satisfies
\[
D_i \left(a^{ij}(x_0',x'')D_{j}\hat v \right)=0\quad \text{in}\quad B_{R/2}(x_0).
\]
Clearly, the equation above still holds with
\[
\tilde v:=\hat v-(\hat v)_{B_{R/2}(x_0)}
\]
in place of $\hat v$.
Take any $\gamma\in (\hd,1)$.
By applying Lemma~\ref{lem3} with a scaling (the modulus of continuity only improves!) and the Poincar\'e inequality, we get (cf. \eqref{eq12.31mnp} above)
\begin{align}
                                \label{eq12.31d}
\int_{B_r(x_0)}\Abs{\hat v-(\hat v)_{B_r(x_0)}}^2 \,dx
&=\int_{B_r(x_0)}\Abs{\tilde v-(\tilde v)_{B_r(x_0)}}^2 \,dx\nonumber\\
&\le Nr^{d+2\gamma}[\tilde v]^2_{\gamma; \,B_{R/4}(x_0)}\le N\left(\frac{r}{R}\right)^{d+2\gamma}\int_{B_{R/2}(x_0)} \abs{\tilde v}^2 \,dx\nonumber\\
&\le N\left(\frac{r}{R}\right)^{d+2\gamma}\int_{B_{R/2}(x_0)}\Abs{\hat v-(\hat v)_{B_{R/2}(x_0)}}^2 \,dx.
\end{align}
Then, similar to \eqref{eq12.43md}, we get from \eqref{eq11.41d} and \eqref{eq12.31d} that
\begin{multline}
                                        \label{eq12.43d}
\int_{B_r(x_0)} \Abs{D_{x'}u-(D_{x'}u)_{B_r(x_0)}}^2 \,dx
\le N\left(\frac{r}{R}\right)^{d+2\gamma}\int_{B_{R}(x_0)} \Abs{D_{x'}u-(D_{x'} u)_{B_{R}(x_0)}}^2 \,dx \\
+NR^{d+2\hd}\left([\vec f]^2_{x',\hd; \,B_1} + [a]_{x', \delta; \, B_1}^2 \norm{Du}^2_{L_p(B_1)}\right).
\end{multline}
By Lemma~\ref{lem:giaq} and \eqref{eq12.43d}, for all $0<r<1/16$ we have
\begin{align*}
\int_{B_r(x_0)} & \Abs{D_{x'}u-(D_{x'}u)_{B_r(x_0)}}^2 dx\nonumber\\
&\le Nr^{d+2 \hd}\left(\int_{B_{1/4}(x_0)} \Abs{D_{x'} u}^2 dx+[\vec f]^2_{x',\hd; B_1}
+[a]_{x', \delta; B_1}^2 \norm{Du}^2_{L_p(B_1)}\right)\nonumber\\
&\le Nr^{d+2 \hd}\left([\vec f]^2_{x',\hd; \,B_1}
+ \left(1+[a]_{x', \delta; \,B_1}^2\right)  \norm{Du}^2_{L_p(B_1)}\right).
\end{align*}
Therefore, by Campanato's theorem, we get \eqref{eq4.12}.
The theorem is proved.
\hfill\qedsymbol

\subsection{Remarks on equations with lower-order terms}
                            \label{sec3.5}
In this subsection, we illustrate how to extend the results in Section \ref{sec:m} to operators with lower-order terms in non-divergence form
\[
Lu:=a^{ij}(x)D_{ij}u+b^i(x) D_iu+c(x)u,
\]
and in divergence form
\[
\cL u:=D_i(a^{ij}(x)D_j u+\hat b^i(x)u)+b^i(x) D_i u+c(x)u,
\]
where the lower-order coefficients are bounded and measurable, which satisfies
$\abs{b^i}$, $\abs{\hat b^i}$, $\abs{c} \le K$ for some $K\ge 0$.
As in the classical Schauder theory, the idea is to move lower-order terms to the right-hand side. However, because we only have estimates of partial H\"older norms, the argument here is a bit more involved.

Theorem \ref{thm1}~(ii) still holds if we assume, for instance, $b=[b^{i}]$ are independent of $x'$, and $c\in C^\delta_{x'}$.
In this case, by using the Krylov--Safonov estimate and the method of the finite difference quotients, it is easily seen that $u\in C^{\delta}_{x'}(B_r)$ for any $r\in (0,1)$, so one can move the zeroth-order term $cu$ to the right-hand of the equation.
Because $b$ are independent of $x'$ and in the proof $\kappa r\le 1$, the first-order term in the equation does not cause any trouble.

In the same fashion, Theorem \ref{thm1}~(iii) can be extended to the case when $b$ are independent of $x'$, and $\hat b=[\hat b^i]\in C^\delta_{x'}$.
Moreover, we can add $g\in L_p(B_1)$ to the right-hand side of the equation, where $p=d/(1-\delta)$.
To see this, first the term $cu$ can be absorbed to $g$. Now let $w\in W^2_p(B_1)$ be the unique solution to the equation $\Delta w=g$ in $B_1$ with the zero Dirichlet boundary condition on $\partial B_1$.
Then by the classical $W^2_p$ estimate and the Sobolev embedding theorem, we have
\[
[Dw]_{\delta;\, B_1}\le N \norm{w}_{W^2_p(B_1)} \le N \norm{g}_{L_p(B_1)},
\]
so we can rewrite the right-hand side as  $\dv (\vec f+\nabla w)$.
Similar to the reasoning above, we have $u\in C^{\delta}_{x'}(B_r)$ for any $r\in (0,1)$.
Thus the term $D_i(\hat b^i u)$ can be absorbed to the right-hand side as well.

In Theorem \ref{thm4m}~(i), we may assume that $b, c\in C^{\hd}_{x'}$.
Indeed, since $1-d/p\ge\delta-d/p\ge \hd$, by the Sobolev embedding theorem, we have $Du$, $u\in C^{\hd}(B_1)$.
Therefore, we can move the lower-order terms to the right-hand side.
Theorem \ref{thm4m}~(ii) can be extended to the case when $\hat b \in C^{\hd}_{x'}$ with an additional term $g\in L_p(B_1)$ on the right-hand side.
By the Sobolev embedding theorem, $u\in C^{\hd}(B_1)$.
Thus, $\hat b^i u$ can be absorbed to $\vec f$, and $b^i Du+cu$ can be absorbed to $g$.
The same extension can be carried out in Theorems \ref{thm4mn} and \ref{thm4}.

\mysection{Main results for parabolic equations}
                        \label{sec:p}

In this section, we consider parabolic operators in non-divergence form
\begin{equation}
				\label{eq0.3}
Pu:=u_t-a^{ij}(t,x) D_{ij}u
\end{equation}
and in divergence form
\begin{equation}
				\label{eq0.4}
\cP u:=u_t- D_i(a^{ij}(t,x)D_j u)
\end{equation}
where $t\in \bR$ and $x=(x',x'')\in \bR^d$.
Here, we assume the coefficients $a^{ij}(t,x)$ are bounded measurable functions on $\bR^{d+1}$ and satisfy the uniform ellipticity condition
\begin{equation}
				\label{parabolic}
\nu \abs{\xi}^2\le a^{ij}(t,x)\xi^i\xi^j\le \nu^{-1} \abs{\xi}^2,\quad\forall (t,x)\in\bR^{d+1},\,\, \xi\in \bR^d,
\end{equation}
for some constant $\nu\in (0,1]$.
As in the elliptic case we assume the symmetry of the coefficients for the non-divergence form operators $P$ but for the operators $\cP$ in divergence form, we instead assume that $\sum_{i,j=1}^d  \abs{a^{ij}(t,x)}^2\le \nu^{-2}$ for all $(t,x) \in \bR^{d+1}$.

For a function $u(t,x)=u(t,x',x'')$ on $Q\subset \bR^{d+1}$, we define a partial H\"older semi-norm with respect to $x'$ as
\[
[u]_{x',\delta; \,Q}:=\!\!\sup_{\substack{(t,x',x''), \,(t,y',x'')\in Q\\x'\neq y'}}\frac {\abs{u(t,x',x'')-u(t,y',x'')}}{\abs{x'-y'}^\delta}.
\]
Similarly, we define the partial H\"older semi-norm with respect to $t$ as
\[
[u]_{t,\delta; \,Q}:=\!\!\sup_{\substack{(t,x), \,(s,x)\in Q\\t\neq s}}\frac {\abs{u(t,x)-u(s,x)}}{\abs{t-s}^\delta},
\]
and the partial H\"older semi-norms with respect to $z':=(t,x')$ as
\[
[u]_{z',\delta/2,\delta; \,Q}:=[u]_{t,\delta/2; \,Q}+[u]_{x',\delta;\, Q},\quad
[u]_{z',(1+\delta)/2,1+\delta; \,Q}:=[u]_{t,(1+\delta)/2; \,Q}+[D_{x'}u]_{z',\delta/2,\delta; \,Q}.
\]
It should be made clear that the above definitions for $[u]_{z',\delta/2,\delta; \,Q}$ and $[u]_{z',(1+\delta)/2,1+\delta; \,Q}$ are equivalent to those defined in \cite{DK11} with $Q=\bR^{d+1}$.
Other related definitions such as $[u]_{x',k+\delta; \,Q}$, $[u]_{t, k+\delta; \,Q}$, and $[u]_{z',(k+\delta)/2,k+\delta;\,Q}$ ($k=0,1,2,\ldots$) are accordingly extended to functions $u=u(t,x)$ on $Q$.
The function spaces $C^{k+\delta}_{x'}(Q)$, $C_t^{k+\delta}(Q)$, and $C^{(k+\delta)/2,\, k+\delta}_{z'}(Q)$ are defined accordingly for $k=0, 1,2, \ldots$.

We say that $u\in W^{1,2}_p(Q)$ for some $p\ge 1$ if $u$ and its weak derivatives $D u$, $D^2 u$, and $u_t$ are in $L_p(Q)$.
For $p\in (1,\infty)$, we say that $u\in W^{1,2}_{p; \,loc}$ is a strong solution of $Pu=f$ if $u$ satisfies the equation $Pu=f$ a.e.

We also denote $\bH^{-1}_{p}(Q)$ to be the space consisting of all functions $u$ satisfying
\[
\inf \set{ \norm{\vec g}_{L_{p}(Q)}+\norm{h}_{L_{p}(Q)}\,:\,u=\dv \vec g+h}<\infty.
\]
It is easy to see that $\bH^{-1}_{p}(Q)$ is a Banach space. Naturally, for any $u\in \bH^{-1}_{p}(Q)$, we define the norm
\[
\norm{u}_{\bH^{-1}_{p}(Q)}=\inf\set{\norm{\vec g}_{L_{p}(Q)}+\norm{h}_{L_{p}(Q)}\,:\,u=\dv \vec g+h}.
\]
We also define
\[
\cH^{1}_{p}(Q)= \set{u:\, u, \,Du \in L_{p}(Q),\,u_t \in \bH^{-1}_{p}(Q)}.
\]
Finally, for $z=(t,x)\in \bR^{d+1}$ and $\rho>0$, let $Q_\rho(z)=(t-\rho^2,t)\times B_\rho(x)$ and $\partial_p Q_\rho(z)$ be its parabolic boundary.
We write $Q_r=Q_r(0)$ for simplicity.
The next theorem is a parabolic counterpart of Theorem~\ref{thm1}.
Hereafter, we denote
\[
\bR^{d+1}_{0} :=(-\infty,0)\times \bR^d.
\]

\begin{theorem}
                                    \label{thm2}

\begin{enumerate}[(i)]
\item
Let $u\in W^{1,2}_{d+1; \,loc}(\bR^{d+1}_0)$ be a bounded strong solution of the equation
\[
P u=f \quad \text{in}\,\,\bR^{d+1}_0.
\]
If $a=[a^{ij}]$ are independent of $x'$ and $f \in C^{\delta}_{x'}(\bR_0^{d+1})$, then $u\in C_{x'}^{2+\delta}(\bR_0^{d+1})$ and there is a constant $N=N(d,q,\delta,\nu)$ such that
\[
[u]_{x',2+\delta;\,\bR^{d+1}_0}\le N[f]_{x',\delta; \,\bR^{d+1}_0}.
\]
If $a=[a^{ij}]$ are independent of $t$ and $f \in C^{\delta/2}_{t}(\bR_0^{d+1})$, then $u\in C_{t}^{1+\delta/2}(\bR_0^{d+1})$ and there is a constant $N=N(d,q,\delta,\nu)$ such that
\[
[u]_{t,1+\delta/2; \,\bR^{d+1}_0}\le N[f]_{t,\delta/2 ;\,\bR^{d+1}_0}.
\]

\item
Let $u\in W^{1,2}_{d+1; \,loc}(Q_1)$ be a bounded strong solution of the equation
\[
P u=f \quad \text{in}\,\,Q_1.
\]
If $a=[a^{ij}]$ are independent of $x'$ and $f \in C^{\delta}_{x'}(Q_1)$, then $u\in C_{x'}^{2+\delta}(Q_{1/2})$ and there is a constant $N=N(d,q,\delta,\nu)$ such that
\[
[u]_{x',2+\delta; \,Q_{1/2}}\le N \left( [f]_{x',\delta;\,Q_1}+ \abs{u}_{0; \,Q_1} \right).
\]
If $a=[a^{ij}]$ are independent of $t$ and $f\in C^{\delta/2}_{t}(Q_1)$,
then $u\in C_{t}^{1+\delta/2}(Q_{1/2})$ and there is a constant $N=N(d,q,\delta,\nu)$ such that
\[
[u]_{t,1+\delta/2; \,Q_{1/2}}\le N \left( [f]_{t,\delta/2;\,Q_1}+ \abs{u}_{0; \,Q_1} \right).
\]
\item
Let $u\in \cH^{1}_{2}(Q_1)$ be a bounded weak solution of the equation
\[
\cP u= \dv \vec f \quad \text{in}\,\,Q_1.
\]
If $a=[a^{ij}]$ are independent of $x'$ and $\vec f \in C^{\delta}_{x'}(Q_1)$, then $u\in C_{x'}^{1+\delta}(Q_{1/2})$ and there is a constant $N=N(d,q,\delta,\nu)$ such that
\[
[u]_{x',1+\delta; \,Q_{1/2}}\le N \left( [\vec f]_{x',\delta; \,Q_1}+ \abs{u}_{0; \,Q_1} \right).
\]
If $a=[a^{ij}]$ are independent of $t$ and $\vec f\in C^{\delta/2}_{t}(Q_1)$,
then $u\in C_{t}^{(1+\delta)/2}(Q_{1/2})$ and there is a constant $N=N(d,q,\delta,\nu)$ such that
\[
[u]_{t,(1+\delta)/2; \,Q_{1/2}}\le N \left( [\vec f]_{t,\delta/2; \,Q_1}+ \abs{u}_{0; \,Q_1} \right).
\]
\end{enumerate}
\end{theorem}

We note that Theorem~\ref{thm2} (i) is an improvement of \cite[Theorem~2.13]{DK11}, where continuity of the coefficients was assumed.
By using an interpolation inequality in Lemma \ref{lem6.1}, Theorem~\ref{thm2} (i) also provides an improvement of \cite[Theorem~2.15]{DK11}.
For the sake of record, we state it as a corollary.
It should be also noted that the same can be said to other results in this section regarding estimates for $[u]_{z'}$.

\begin{corollary}				\label{cor4.5}
Let $u\in W^{1,2}_{d+1; \,loc}(\bR_0^{d+1})$ be a bounded strong solution of the equation
\[
P u=f \quad \text{in}\,\,\bR_0^{d+1}.
\]
If $a=[a^{ij}]$ are independent of $z'$ and $f\in C^{\delta/2,\delta}_{z'}(\bR_0^{d+1})$, then $u\in C_{z'}^{1+\delta/2,2+\delta}(\bR_0^{d+1})$ and there is a constant $N=N(d,q,\delta,\nu)$ such that
\[
[u]_{z',1+\delta/2,2+\delta; \,\bR^{d+1}_0}\le N[f]_{z',\delta/2,\delta; \,\bR^{d+1}_0}.
\]
\end{corollary}

Next we consider parabolic equations in non-divergence form \eqref{eq0.3} and in divergence form \eqref{eq0.4} with coefficients depending on all the variables.

\begin{theorem}
                                    \label{thm5m}
Let $\delta\in (0,1]$ and $p\in (d+2,\infty)$ be such that $\delta-(d+2)/p>0$. Then, there exists a constant $\delta_0=\delta_0(d,\nu)>0$ such that the following assertions hold with any $\hd\in (0,\delta_0)$ satisfying $\hd\le \delta-(d+2)/p$.

\begin{enumerate}[(i)]
\item
Let $u\in W^{1,2}_{p}(Q_1)$ be a strong solution of the equation
\[
P u=f \quad \text{in }\,Q_1.
\]
If $a=[a^{ij}]$ are $\delta$-H\"older continuous in $x'$ and $f \in C^{\hd}_{x'}(Q_1)$, then $D^2_{x'}u\in C^{\hd/2,\hd}(Q_{1/2})$ and there is a constant $N=N(d, q, \nu, \delta,p)$ such that
\begin{equation}
                                            \label{eq4.11pm}
[D^2_{x'}u]_{\hd/2,\hd; \,Q_{1/2}}\le N \left( [f]_{x',\hd; \,Q_1}+ \left(1+ [a]_{x', \delta; \,Q_1}\right) \norm{D^2 u}_{L_p(Q_1)} \right).
\end{equation}
If $a=[a^{ij}]$ are $\delta/2$-H\"older continuous in $t$ and $f\in C^{\hd/2}_{t}(Q_1)$, then $u_t\in C^{\hd/2,\hd}(Q_{1/2})$
and there is a constant $N=N(d, \nu, \delta,p)$ such that
\begin{equation}
                                            \label{eq4.11pmz}
[u_t]_{\hd/2,\hd; \,Q_{1/2}}\le N \left( [f]_{t,\hd/2; \,Q_1}+ \norm{f}_{L_1(Q_1)} + \left(1+ [a]_{t,\delta/2; \,Q_1}\right) \norm{D^2 u}_{L_p(Q_1)} \right).
\end{equation}

\item
Let $u\in \cH^1_p(Q_1)$ be a weak solution of the equation
\[
\cP u=\dv \vec f \quad \text{in }\,Q_1.
\]
If $a=[a^{ij}]$ are $\delta$-H\"older continuous in $x'$ and $\vec f \in C^{\hd}_{x'}(Q_1)$, then $D_{x'}u\in C^{\hd/2, \hd}(Q_{1/2})$ and there is a constant $N=N(d, q, \nu, \delta, p)$ such that
\begin{equation}
                                            \label{eq4.12pm}
[D_{x'}u]_{\hd/2,\hd;\,Q_{1/2}}\le N \left( [\vec f]_{x',\hd;\,Q_1}+\left(1+ [a]_{x', \delta;\,Q_1}\right)\norm{Du}_{L_p(Q_1)} \right).
\end{equation}
If $a=[a^{ij}]$ are $\delta/2$-H\"older continuous in $t$ and $\vec f\in C^{\hd/2}_{t}(Q_1)$, then $u\in C^{(1+\hd)/2}_t(Q_{1/2})$ and there is a constant $N=N(d, \nu, \delta, p)$ such that
\begin{equation}
                                            \label{eq4.12pmz}
[u]_{t,(1+\hd)/2;\,Q_{1/2}}\le N \left( [\vec f]_{t,\hd/2;\,Q_1}+\norm{\vec f}_{L_{1}(Q_{1})}+\left(1+ [a]_{t,\delta/2;\,Q_1}\right)\norm{Du}_{L_p(Q_1)} \right).
\end{equation}
For \eqref{eq4.12pmz}, the condition $\hd\in (0,\delta_0)$ is not needed.
\end{enumerate}
\end{theorem}

\begin{theorem}
                                    \label{thm5}
Let $\delta\in (0,1]$ and $p\in (d+2,\infty)$ be such that $\hd:=\delta-(d+2)/p>0$.
Assume that $a=[a^{ij}]$ are uniformly continuous in $(t,x^1,\ldots,x^{d-1})$ and merely measurable in $x^d$.
Let $\omega_{a}$ denote a modulus of continuity of $a=[a^{ij}]$ with respect to $(t,x^1,\ldots, x^{d-1})$.

\begin{enumerate}[(i)]
\item
Let $u\in W^{1,2}_{p}(Q_1)$ be a strong solution of the equation
\[
P u=f \quad \text{in }\,Q_1.
\]
If $a=[a^{ij}]$ are $\delta$-H\"older continuous in $x'$ and $f \in C^{\hd}_{x'}(Q_1)$,
then we have $D_{x'}u\in C^{(1+\hd)/2,1+\hd}(Q_{1/2})$ and there is a constant $N$ depending on $d$, $p$, $\delta$, $\nu$, and $\omega_a$, such that
\begin{equation}
                                            \label{eq4.11p}
[D_{x'}u]_{(1+\hd)/2,1+\hd;\,Q_{1/2}}
\le N \left([f]_{x',\hd;\,Q_1}+  \left(1+ [a]_{x', \delta;\,Q_1}\right)\norm{D^2 u}_{L_p(Q_1)} \right).
\end{equation}
If $a=[a^{ij}]$ are $\delta/2$-H\"older continuous in $t$ and $f\in C^{\hd/2}_{t}(Q_1)$, then $u_t\in C^{\hd/2,\hd}(Q_{1/2})$  and there is a constant $N$ depending on $d$, $p$, $\delta$, $\nu$, and $\omega_a$, such that
\begin{equation}
                                            \label{eq4.11pz}
[u_t]_{\hd/2,\hd;\,Q_{1/2}}\le N \left([f]_{t,\hd/2;\,Q_1}+ \norm{f}_{L_1(Q_1)}+ \left(1+ [a]_{t,\delta/2;\,Q_1}\right)\norm{D^2 u}_{L_p(Q_1)} \right).
\end{equation}

\item
Let $u\in \cH^1_p(Q_1)$ be a weak solution of the equation
\[
\cP u=\dv \vec f \quad \text{in }\,Q_1.
\]
If $a=[a^{ij}]$ are $\delta$-H\"older continuous in $x'$ and $\vec f \in  C^{\hd}_{x'}(Q_1)$, then $D_{x'}u\in C^{\hd/2,\hd}(Q_{1/2})$ and there is a constant $N$ depending on $d$, $p$, $\delta$, $\nu$, and $\omega_a$, such that
\[
[D_{x'} u]_{\hd/2,\hd;\,Q_{1/2}}\le N \left([\vec f]_{x',\hd;\,Q_1}+\left(1+ [a]_{x',\delta;\,Q_1}\right) \norm{Du}_{L_p(Q_1)} \right).
\]
\end{enumerate}
\end{theorem}

\begin{remark}
In Theorem \ref{thm5}, we may assume that $a^{ij}$ are uniformly continuous in $x$ instead of $(t,x^1,\ldots,x^{d-1})$.
All that is needed is $W^{1,2}_p$ solvability of the equation, and as such this condition can be relaxed to the vanishing mean oscillation (VMO) condition or partially VMO condition; see, for instance, \cite{KK07} and \cite{Dong12b}.
\end{remark}

\mysection{The proofs: Parabolic estimates}		\label{sec:pp}

\subsection{Proof of Theorem~\ref{thm2}}				\label{sec5.1}
The proof is similar to that of Theorem~\ref{thm1} and only minor adjustments are needed.
We begin with proving assertion (i).
We consider the both cases (i.e., the cases when $a$ are independent of either $x'$ or $t$) simultaneously.

For a function $v$ defined on $\bR_0^{d+1}$ and $\epsilon>0$, we define a partial mollification of $v$ with respect to $x'$ as
\[
\tv^\epsilon(t,x',x''):=\int_{\bR^q}v(t,x'-\epsilon y',x'')\zeta(y')\,dy'
\]
and a partial mollification with respect to $t$ as
\[
\tv^\epsilon(t,x):=\int_{0}^\infty\big(2v(t-\epsilon^2 s,x)-v(t-2\epsilon^2 s,x)\big) \eta(s-1)\,ds,
\]
where $\eta$ and $\zeta$ are as defined in Section~\ref{sec3.1}.
The above definition enables us to obtain an analogue of Lemma~\ref{lem12.25}.
In particular, we have
\[
\abs{v(z_0) -\tv^\epsilon (z_0)} \le  N \epsilon^{2+\delta}[v]_{t, 1+\delta/2;\, Q_\epsilon(z_0)}.
\]
Notations regarding partial Taylor polynomials such as $\tT^k_{x_0'} v$ and $\tT^k_{t_0} v$ are defined in an obvious way.
Similar to the proof of Theorem~\ref{thm1}, we may assume $u \in C^{2+\delta}_{x'}(\bR^{d+1}_0)$ (resp. $u \in C^{1+\delta/2}_{t}(\bR^{d+1}_0)\,$) and let $a_n=[a^{ij}_n]$ be a sequence of coefficients that are continuous, independent of $x'$ (resp. independent of $t$), satisfy \eqref{parabolic}, and $a_n \to a$ a.e. as $n\to \infty$.
Let $P_n$ be the corresponding operator with $a_n$ in place of $a$.
Then
\[
P_n u=f_n,\quad \text{where}\quad f_n=f-(a^{ij}_n-a^{ij})D_{ij}u.
\]
Let $\kappa>2$ be a number to be chosen later.
Then, we have for any $r>0$,
\[
P_n \tu^{\kappa r}=\tf_n^{\kappa r},
\]
where $\tu^\epsilon$ is a partial mollification with respect to $x'$ (resp. with respect to $t$).
Let $Q_{r}=Q_{r}(z_0)$, where $z_0$ is a point in $\bR^{d+1}$, and let $w\,(=w_n)\in W^2_{d+1;\,loc}(Q_{\kappa r})\cap C^0(\overline Q_{\kappa r})$ be a unique solution of the problem (see \cite[Theorem 7.17]{Lieberman})
\[
\left\{
  \begin{aligned}
    P_n w = 0 \quad & \hbox{in $\;Q_{\kappa r}$,} \\
    w=u-\tu^{\kappa r} \quad & \hbox{on $\;\partial_p Q_{\kappa r}$.}
  \end{aligned}
\right.
\]
By the ABP maximum principle and an analogue of Lemma \ref{lem12.25} (ii), we obtain
\begin{equation}
                                                        \label{eq1.27p}
\sup_{Q_{\kappa r}}\, \abs{w}= \sup_{\partial_p Q_{\kappa r}} \,\abs{w}\le N(\kappa r)^{2+\delta}[u]_{x', 2+\delta; \,\bR^{d+1}_0} \quad\bigl(\text{resp. }\le N(\kappa r)^{2+\delta}[u]_{t,1+\delta/2; \,\bR^{d+1}_0}\bigr).
\end{equation}
By the Krylov--Safonov theorem, $w\in C^{\delta_0/2,\delta_0}_{loc}(Q_{\kappa r})$ for some $\delta_0=\delta_0(d,\nu)\in (0,1)$.
Since $a_n^{ij}$ are independent of $x'$ (resp. independent of $t$), for any integer $j \ge 0$, there is a constant $N=N(d,q, \nu, j)$ such that we have (cf. \eqref{eq16.41})
\begin{equation}
				\label{eq2.76a}
\abs{D^j_{x'} w}_{0; \,Q_{\kappa r/2}}\le N\,(\kappa r)^{-j}  \abs{w}_{0;\,Q_{\kappa r}}\;\; \bigl(\;\text{resp. }\,  \abs{D^j_t w}_{0;\,Q_{\kappa r/2}}\le N\,(\kappa r)^{-2j}  \abs{w}_{0;\,Q_{\kappa r}}\,\bigr)
\end{equation}
Notice that Taylor's formula yields (see \cite[Theorem 8.6.1]{Kr96})
\begin{equation}
				\label{eq2.60q}
\abs{w-\tT^2_{x_0'}w}_{0;\,Q_r}  \le Nr^3 \abs{D^3_{x'} w}_{0;\,Q_r}\;\;
\bigl(\;\text{resp. }\,\abs{w-\tT^1_{t_0}w}_{0; \,Q_r}  \le  N r^4 \abs{D^2_t w}_{0;\,Q_r}\;\bigr).
\end{equation}
Then we obtain from \eqref{eq2.60q}, \eqref{eq2.76a}, and \eqref{eq1.27p}
\begin{align*}
\abs{w-\tT^2_{x_0'}w}_{0;\,Q_r} & \le N \kappa^{-3} \abs{w}_{0;\,Q_{\kappa r}} \le N\kappa^{\delta-1} r^{2+\delta} [u]_{x',2+\delta; \,\bR^{d+1}_0}. \\
\bigl(\;\text{resp. }\; \abs{w-\tT^1_{t_0}w}_{0;\,Q_r} &\le N \kappa^{-4} \abs{w}_{0;\,Q_{\kappa r}}\le N \kappa^{\delta-2} r^{2+\delta} [u]_{t,1+\delta/2; \,\bR^{d+1}_0} \;\bigr).
\end{align*}
On the other hand, $v:=u-\tu^{\kappa r}-w$ satisfies
\[
\left\{
  \begin{aligned}
    P_n v = f_n-\tf_n^{\kappa r} \quad & \hbox{in $Q_{\kappa r}$,} \\
    v=0 \quad & \hbox{on $\partial_p Q_{\kappa r}$.}
  \end{aligned}
\right.
\]
Therefore, we have (similar to the derivation of \eqref{eq2.14})
\[
\abs{u-\tu^{\kappa r}-w}_{0;\,Q_{\kappa r}} \le N(\kappa r)^{2+\delta} [f]_{x',\delta;\,\bR^{d+1}_0}+ N E\;\;
\bigl(\;\text{resp. }\; \le N(\kappa r)^{2+\delta} [f]_{t,\delta/2; \,\bR^{d+1}_0}+ N E\;\bigr),
\]
where we set $E=\norm{(a^{ij}_n-a^{ij})D_{ij}u}_{L_{d+1}(Q_{2\kappa r})}$, which tends to zero as $n \to \infty$ by dominated convergence theorem.
Also, similar to \eqref{eq2.17}, we get
\begin{align*}
\abs{\tu^{\kappa r}-\tT^2_{x_0'}\tu^{\kappa r}}_{0; \,Q_r} &\le N \kappa^{\delta-1} r^{2+\delta} [u]_{x',2+\delta;\bR^{d+1}_0}.\\
\bigl(\;\text{resp. }\;\abs{\tu^{\kappa r}-\tT^1_{t_0}\tu^{\kappa r}}_{0; \,Q_r} &\le  N \kappa^{\delta-2} r^{2+\delta} [u]_{t, 1+\delta/2;\bR^{d+1}_0}.\;\bigr)
\end{align*}
Take $p=\tT^2_{x_0'}w+\tT^2_{x_0'} \tu^{\kappa r}$
(resp. $p=\tT^1_{t_0}w+\tT^1_{t_0} \tu^{\kappa r}$).
Then similar to \eqref{eq3.09d}, we have
\begin{align*}
\abs{u-p}_{0; \,Q_r(z_0)} &\le N\kappa^{\delta-1} r^{2+\delta}[u]_{x',2+\delta; \,\bR^{d+1}_0}+N(\kappa r)^{2+\delta} [f]_{x',\delta; \,\bR^{d+1}_0}+N E.\\
\bigl(\;\text{resp. }\;\abs{u-p}_{0; \,Q_r(z_0)} &\le N \kappa^{\delta-2} r^{2+\delta}[u]_{t,1+\delta/2;\,\bR^{d+1}_0}+N(\kappa r)^{2+\delta} [f]_{t,\delta/2; \,\bR^{d+1}_0}+NE.\;\bigr)
\end{align*}
Letting $n\to \infty$, this implies
\begin{align*}
r^{-2-\delta}\,\inf_{p\in\tilde\bP_{2}} \abs{u-p}_{0; \,Q_r(z_0)} &\le N\kappa^{\delta-1} [u]_{x',2+\delta; \,\bR^{d+1}_0}+N \kappa^{2+\delta} [f]_{x',\delta; \,\bR^{d+1}_0}\\
\bigl(\text{resp.}\quad r^{-2-\delta}\,\inf_{p\in\tilde\bP_{1}} \abs{u-p}_{0; \,Q_r(z_0)}&\le N\kappa^{\delta-2} [u]_{t,1+\delta/2; \,\bR^{d+1}_0}+N \kappa^{2+\delta} [f]_{t,\delta/2; \,\bR^{d+1}_0}\bigr)
\end{align*}
for any $z_0\in \bR^{d+1}$ and $r>0$.
We take the supremum of the above with respect to $z_0\in \bR^{d+1}$ and $r>0$, and then apply \cite[Theorem 3.3.1]{Kr96} to get
\begin{align*}
[u]_{x',2+\delta; \,\bR^{d+1}_0} &\le N\kappa^{\delta-1} [u]_{x',2+\delta; \,\bR^{d+1}_0}+N\kappa^{2+\delta} [f]_{x',\delta; \,\bR^{d+1}_0}.\\
\bigl(\;\text{resp.}\quad [u]_{t,1+\delta/2; \,\bR^{d+1}_0} &\le N\kappa^{\delta-2} [u]_{t,1+\delta/2; \,\bR^{d+1}_0}+N\kappa^{2+\delta} [f]_{t,\delta/2; \,\bR^{d+1}_0}.\;\bigr)
\end{align*}
To finish the proof of assertion (i), it suffices to choose a large $\kappa$ such that $N\kappa^{\delta-1}<1/2$.
This completes the proof of assertion (i).

Assertion (ii) then follows by combining the proof of assertion (i) with that of Theorem \ref{thm1}~(ii) (with cylinders in place of balls).
Finally, assertion (iii) is obtained similarly by a minor modification of the proof above (cf. \cite[Theorem 2.16]{DK11}).
We leave the details to the interested reader.
The theorem is proved.
\hfill\qedsymbol

\subsection{Proof of Theorem~\ref{thm5m}}
The following lemma is a parabolic version of Lemma \ref{lem4}, the proof of which can be found in \cite[Corollary 4.2]{Kr10} and \cite[Lemma~5.5]{DoKrLi}.
\begin{lemma}
                                                  \label{lem4p}
Let $w\in   W^{1,2}_{d+1}(Q_{r})$ be a function such that $w=0$ on $\partial_p Q_r$. Then there are constants $\epsilon \in (0,1]$ and $N$,
depending only on $d$ and $\nu$, such that we have
\[
\fint_{Q_{r}} \abs{D^2 w}^\epsilon \,dx \,dt   \leq N \left(\fint_{Q_{r}} \abs{P w}^{d+1} \,dx \,dt \right)^{\epsilon/(d+1)}.
\]
\end{lemma}

\subsubsection{Proof of assertion (i)}
Similar to the proof of Theorem~\ref{thm4m}, we may assume that $D^2_{x'}u\in C^{\hd/2,\hd}(Q_{1/2})$ (resp. $u_t\in C^{\hd/2,\hd}(Q_{1/2})$).
Moreover, we can find a sequence of continuous coefficients $a_n=[a^{ij}_n]$, which are $\delta$-H\"older continuous in $x'$ (resp. $\delta/2$-H\"older continuous in $t$)  with $[a_n]_{x',\delta}\le [a]_{x',\delta}$ (resp. $[a_n]_{t,\delta/2}\le [a]_{t,\delta/2}$), satisfy \eqref{parabolic}, and $a_n\to a$ a.e. as $n\to \infty$.
Let $P_n$ be the corresponding operator with $a_n$ in place of $a$.
Then we have
\[
P_n u=f_n,\quad \text{where}\quad f_n=f-(a^{ij}_n-a^{ij})D_{ij}u.
\]
We take a point $z_0\in Q_{1/2}$ and $r, R\in (0,1/4)$ such that $0<r<R/4$.
By the classical $W^{1,2}_{d+1}$ solvability for parabolic equations with continuous coefficients, there is a unique solution $w\in W^{1,2}_{d+1}(Q_{R}(z_0))$ of the equation
\begin{gather*}
w_t-a_n^{ij}(t, x_0',x'')D_{ij}w=f_n-f(t, x_0',x'')-\left(a_n^{ij}(t, x_0',x'')-a_n^{ij}\right)D_{ij}u \\
\bigl(\;\text{resp.} \quad w_t-a_n^{ij}(t_0, x)D_{ij}w=f_n-f(t_0, x)-\left(a_n^{ij}(t_0, x)-a_n^{ij}\right)D_{ij}u \; \bigr)
\end{gather*}
in $Q_{R}(z_0)$ with zero Dirichlet boundary value on $\partial_p Q_{R}(z_0)$.
Thanks to Lemma \ref{lem4p}, we have (similar to the derivation of \eqref{eq11.41m})
\begin{multline*}
\int_{Q_R(z_0)} \Abs{D^2 w}^\epsilon \le N R^{d+2-\epsilon(d+2)/(d+1)} \norm{f_n-f}_{L_{d+1}(Q_R(z_0))}^\epsilon\\
+NR^{d+2+\epsilon\hd} [f]^\epsilon_{x',\hd; \,Q_1}+N R^{d+2+\epsilon(\delta-(d+2)/p)} [a]_{x', \delta; \,Q_1}^\epsilon \norm{D^2 u}_{L_p(Q_1)}^\epsilon
\end{multline*}
and, respectively,
\begin{multline*}
\int_{Q_R(z_0)} \Abs{w_t}^\epsilon \le N R^{d+2-\epsilon(d+2)/(d+1)} \norm{f_n-f}_{L_{d+1}(Q_R(z_0))}^\epsilon\\
+NR^{d+2+\epsilon\hd} [f]^\epsilon_{t,\hd/2; \,Q_1}+N R^{d+2+\epsilon(\delta-(d+2)/p)} [a]_{t, \delta/2; \,Q_1}^\epsilon \norm{D^2 u}_{L_p(Q_1)}^\epsilon,
\end{multline*}
where $N=N(d, \nu)$.
It is easily seen that $v:=u-w\in  W^{1,2}_{d+1}(Q_{R}(z_0))$ satisfies
\begin{gather}
v_t-a^{ij}_n(t,x_0',x'')D_{ij}v=f(t,x_0',x'')\quad \text{in}\quad Q_{R}(z_0).\nonumber\\
					\label{eq11.47mp}
\left(\;\text{resp.} \quad v_t-a_n^{ij}(t_0, x) D_{ij}v=f(t_0, x)\quad \text{in}\quad Q_{R}(z_0).\;\right)
\end{gather}
Without loss of generality, we may assume that $v$ is smooth with respect to $x'$ (resp. with respect to $t$).
By differentiating \eqref{eq11.47mp} with respect to $x'$ twice (resp. with respect to $t$ once), we see that $\hat v:=D_{x'}^2 v$ (resp. $\hat v:= v_t$) satisfies
\[
\hat v_t-a_n^{ij}(t, x_0',x'')D_{ij}\hat v=0 \quad
\bigl(\;\text{resp. } \; \hat v_t-a_n^{ij}(t_0, x) D_{ij}\hat v=0\;\bigr) \quad \text{in}\quad Q_{R/2}(z_0).
\]
Clearly, for any constant $c \in \bR$, the same equation is satisfied by $\tilde v:=\hat v-c$ in place of $\hat v$.
By applying the Krylov--Safonov estimate, we get
\begin{multline*}
\int_{Q_r(z_0)} \Abs{\hat v-(\hat v)_{Q_r(z_0)}}^\epsilon
=\int_{Q_r(z_0)} \Abs{\tilde v-(\tilde v)_{Q_r(z_0)}}^\epsilon
\le Nr^{d+2+\epsilon\delta_0}[\tilde v]^\epsilon_{\delta_0/2, \delta_0; \,Q_{R/4}(z_0)}\\
\le N\left(\frac{r}{R}\right)^{d+2+\epsilon\delta_0}\int_{Q_{R/2}(z_0)} \Abs{\tilde v}^\epsilon
=N\left(\frac{r}{R}\right)^{d+2+\epsilon\delta_0}\int_{Q_{R/2}(z_0)} \Abs{\hat v-c}^\epsilon,
\end{multline*}
for some $\delta_0=\delta_0(d,\nu)>0$.
We set ($D^2_{x'}u=D_{ij}u$ for $i,j=1,\ldots, q$)
\[
\phi(z_0,r):=\inf_{c \in \bR} \int_{Q(z_0,r)} \abs{D^2_{x'}u - c}^\epsilon \quad
\bigl(\;\text{resp. } \; \phi(z_0,r):=\inf_{c \in \bR} \int_{Q(z_0,r)} \abs{u_t - c}^\epsilon\;\bigr).
\]
Then, we get (similar to the derivation of \eqref{eq12.43m1})
\begin{multline}		                                        \label{eq12.43m1p}
\phi(z_0,r) \le N\left(\frac{r}{R}\right)^{d+2+\epsilon\delta_0} \phi(x_0, R)\\
+NR^{d+2+\epsilon\hd} [f]^\epsilon_{x',\hd; \,Q_1}+NR^{d+2+\epsilon(\delta-(d+2)/p)} [a]_{x', \delta; \,Q_1}^\epsilon \norm{D^2 u}_{L_p(Q_1)}^\epsilon
\end{multline}
and, respectively,
\begin{multline}
                                        \label{eq12.43m1pt}
\phi(z_0,r) \le N\left(\frac{r}{R}\right)^{d+2+\epsilon\delta_0} \phi(z_0,R) \\
+NR^{d+2+\epsilon\hd} [f]^\epsilon_{t,\hd/2; \,Q_1}+NR^{d+2+\epsilon(\delta-(d+2)/p)} [a]_{t, \delta/2; \,Q_1}^\epsilon \norm{D^2 u}_{L_p(Q_1)}^\epsilon.
\end{multline}
By Lemma~\ref{lem:giaq}, we get from \eqref{eq12.43m1p} and \eqref{eq12.43m1pt}, respectively, that for any $z_0 \in Q_{1/2}$ and any $r \in (0,1/16)$, there is a constant $c_{z_0,r}$ such that (cf. \eqref{eq2.31m})
\[
\int_{Q_r(z_0)} \Abs{D_{x'}^2 u-c_{z_0, r}}^\epsilon \le Nr^{d+2+\epsilon\hd} \left( [f]^\epsilon_{x',\hd; \,Q_1}
+\left(1+ [a]_{x', \delta;\,Q_1}^\epsilon \right) \norm{D^2 u}^\epsilon_{L_p(Q_1)}\right)
\]
and, respectively,
\[
\int_{Q_r(z_0)} \Abs{u_t-c_{z_0,r}}^\epsilon  \le Nr^{d+2+\epsilon\hd} \left( [f]^\epsilon_{t,\hd/2; \,Q_1} + \norm{f}_{L_1(Q_1)}^\epsilon +\left(1+ [a]_{t,\delta/2; \,Q_1}^\epsilon\right) \norm{D^2 u}^\epsilon_{L_p(Q_1)} \right).
\]
Here, we used
\[
\sup\set{ \phi(z_0,r) : z_0 \in Q_{1/2},\; 0<r<1/16} \le N \norm{D^2 u}_{L_1(Q_1)}^\epsilon \quad \bigl(\;\text{resp. }  \le  N \norm{u_t}_{L_1(Q_1)}^\epsilon \;\bigr)
\]
and
\[
\norm{u_t}_{L_1(Q_1)} \le N \norm{D^2 u}_{L_1(Q_1)} + \norm{f}_{L_1(Q_1)}.
\]
Therefore, we obtain \eqref{eq4.11pm} and \eqref{eq4.11pmz} from the above inequalities combined with Lemma~\ref{lemma6.14}.
This completes the proof of assertion (i).
\hfill\qedsymbol

\subsubsection{Proof of assertion (ii)}
We first consider the case when $a$ is $\delta$-H\"older continuous in $x'$ and $\vec f \in C^{\hd}_{x'}(Q_1)$.
Take a point $z_0\in Q_{1/2}$ and denote
\[
\cP_0 u = u_t - D_i\left(a^{ij}(t,x_0',x'')D_{j}u \right).
\]
Then we have
\[
\cP_0 u = \dv \vec f -D_i \left((a^{ij}(t,x_0',x'')-a^{ij})D_{j}u\right).
\]
Take $r, R\in (0,1/4)$ such that $0<r<R/4$.
Let $w\in \cH^{1}_{2}(Q_{R}(z_0))$ be the weak solution of the equation
\[
\cP_0 w =\dv \left(\vec f-\vec f(t,x_0',x'')\right) -D_i \left((a^{ij}(t,x_0',x'')-a^{ij})D_{j}u\right)
\]
in $Q_{R}(z_0)$ with the zero Dirichlet boundary condition on $\partial_p Q_R(z_0)$.
By using the energy inequality, we have, similar to \eqref{eq11.41md}, that
\begin{align*}
\int_{Q_R(z_0)}  \abs{Dw}^2 \,dx \,dt &\le N\int_{Q_R(z_0)} \Abs{\vec f-\vec f(t,x_0',x'') - \left(a^{ij}(t,x_0',x'')-a^{ij}\right) D_{j}u}^2 \,dx\,dt\\
&\le N \left(R^{d+2+2\hd} [\vec f]^2_{x',\hd; \,Q_1}+R^{d+2+2(\delta-d/p)} [a]_{x', \delta; \,Q_1}^2 \norm{Du}_{L_{p}(Q_1)}^2\right),
\end{align*}
where $N=N(d, \nu)$.
It is easily seen that $v:=u-w\in  \cH^1_2(Q_{R}(x_0))$ satisfies
\begin{equation}
                            \label{eq11.47mdp}
\cP_0 v =\dv \vec f(t,x_0',x'')\quad \text{in}\quad Q_{R}(z_0).
\end{equation}
Without loss of generality, we may assume that $v$ is smooth with respect to $x'$.
By differentiating \eqref{eq11.47mdp} with respect to $x'$, we see that $\hat v:=D_{x'} v$  satisfies
\[
\cP_0 \hat{v}=0\quad \text{in}\quad Q_{R/2}(z_0).
\]
Clearly, the above equation is still satisfied by $\tilde v:=\hat v-(\hat v)_{Q_{R}(z_0)}$ in place of $\hat v$.
Therefore, by applying the De Giorgi--Nash--Moser estimate, we get
\begin{multline*}
\int_{Q_r(z_0)}  \Abs{\hat v-(\hat v)_{Q_r(z_0)}}^2
=\int_{Q_r(z_0)}\Abs{\tilde v-(\tilde v)_{Q_r(z_0)}}^2
\le Nr^{d+2+2\delta_0} [\tilde v]^2_{\delta_0/2, \delta_0; \,Q_{R/4}(z_0)}\\
\le N\left(\frac{r}{R}\right)^{d+2+2\delta_0}\int_{Q_{R/2}(z_0)} \Abs{\tilde v}^2
= N\left(\frac{r}{R}\right)^{d+2+2\delta_0}\int_{Q_{R/2}(z_0)} \Abs{\hat v-(\hat v)_{Q_{R}(z_0)}}^2,
\end{multline*}
where  $\delta_0=\delta_0(n, \nu)>0$ is the H\"older exponent appearing in the De Giorgi--Nash--Moser estimate.
Then, we obtain (similar to the derivation of \eqref{eq12.43md})
\begin{multline}
					\label{eq12.43mdp}
\int_{Q_r(z_0)} \Abs{D_{x'}u-(D_{x'}u)_{Q_r(z_0)}}^2 \le N\left(\frac{r}{R}\right)^{d+2+2\delta_0}\int_{Q_{R}(z_0)} \Abs{D_{x'} u-(D_{x'} u)_{Q_{R}(z_0)}}^2\\
+NR^{d+2+2\hd} \left([\vec f]^2_{x',\hd; \,Q_1} + [a]_{x', \delta; \,Q_1}^2 \norm{Du}^2_{L_p(Q_1)}\right).
\end{multline}
By Lemma~\ref{lem:giaq}, we infer from \eqref{eq12.43mdp} that, for all $0<r<1/16$, we have (cf. \eqref{eq2.31md})
\begin{equation}
					\label{eq2.31mdp}
\int_{Q_r(z_0)} \Abs{D_{x'}u-(D_{x'}u)_{Q_r(z_0)}}^2
\le Nr^{d+2+2\hd}\left([\vec f]^2_{x',\hd; \,Q_1} + \left(1+ [a]_{x', \delta; \,Q_1}^2\right) \norm{Du}^2_{L_p(Q_1)}\right).
\end{equation}
Then we get \eqref{eq4.12pm} from \eqref{eq2.31mdp} by Campanato's theorem.

Next, we consider the case when $a$ is $\delta/2$-H\"older continuous in $t$ and $\vec f \in C^{\hd}_{t}(Q_1)$.
We use the idea in the proofs of Theorems \ref{thm1} and \ref{thm2}.
Let us momentarily assume that $u \in C^{(1+\hd)/2}_{t}$.
Let $\tu^\epsilon$ be a partial mollification with respect to $t$ as defined in Section~\ref{sec5.1}.

For $n=1,2,\ldots$, denote $r_n=3/4-2^{-n-1}$ and $Q^{(n)}=Q_{r_n}$. Now we fix a point $z_0\in Q^{(n)}$.
Let $\kappa>2$ be a number to be fixed later.
For any $r\le 2^{-n-3}/\kappa$, we have $Q_{2 \kappa r}(z_0)\subset Q^{(n+1)}$.
Denote
\[
\cP_0 u = u_t - D_i\left(a^{ij}(t_0,x)D_{j}u \right).
\]
Let $w\in \cH^1_{2}(Q_{\kappa r}(z_0))$ be a weak solution of the problem
\[
\left\{
  \begin{aligned}
    \cP_0 w = 0 \quad & \hbox{in $\;Q_{\kappa r}(z_0)$,} \\
    w=u-\tu^{\kappa r} \quad & \hbox{on $\;\partial_p Q_{\kappa r}(z_0)$.}
  \end{aligned}
\right.
\]
By the weak maximum principle, similar to \eqref{eq1.27p} we obtain
\begin{equation}
                                                        \label{eq1.27pz}
\sup_{Q_{\kappa r}(z_0)}\, \abs{w}= \sup_{\partial_p Q_{\kappa r}(z_0)} \,\abs{w}\le N(\kappa r)^{1+\hd}[u]_{t,(1+\hd)/2; \,Q^{(n+1)}}.
\end{equation}
Then similar to \eqref{eq2.76a} and \eqref{eq2.60q}, we have
\begin{equation}
				\label{eq2.76az}
\abs{D_t w}_{0; \,Q_{\kappa r/2}(z_0)}\le N\,(\kappa r)^{-2}  \abs{w}_{0; \,Q_{\kappa r}(z_0)},
\quad
\abs{w-\tT^0_{t_0}w}_{0; \,Q_r(z_0)}  \le  N r^2 \abs{D_t w}_{0; \,Q_r(z_0)}.
\end{equation}
We obtain from \eqref{eq2.76az} and \eqref{eq1.27pz} that
\begin{align*}
\abs{w-\tT^0_{t_0}w}_{0; \,Q_r(z_0)} &\le N \kappa^{-2} \abs{w}_{0; \,Q_{\kappa r}(z_0)}\le N \kappa^{\hd-1} r^{1+\hd} [u]_{t,(1+\hd)/2; \,Q^{(n+1)}}.
\end{align*}
On the other hand, $v:=u-\tu^{\kappa r}-w$ satisfies
\[
\left\{
  \begin{aligned}
    \cP_0 v =\dv(\vec g-\tilde{\vec g}^{\kappa r}) \quad & \hbox{in $\;Q_{\kappa r}(z_0)$,} \\
    v=0 \quad & \hbox{on $\;\partial_p Q_{\kappa r}(z_0)$,}
  \end{aligned}
\right.
\]
where
\[
\vec g=(g^1,\ldots,g^d)\quad\text{and}\quad g^i=f^i+\left(a^{ij}-a^{ij}(t_0,x)\right)D_j u.
\]
Therefore, by the De Giorgi--Nash--Moser estimate, we have
\begin{align*}
\abs{u-\tu^{\kappa r}-w}_{0; \,Q_{\kappa r}(z_0)}
&\le N(\kappa r)^{1+\hd} [\vec f]_{t,\hd/2; \,Q_1}+ N (\kappa r)^{1-(d+2)/p} \Norm{\left(a^{ij}-a^{ij}(t_0,x)\right)D_{j}u}_{L_{p}(Q_{2\kappa r}(z_0))}\\
&\le N(\kappa r)^{1+\hd} [\vec f]_{t,\hd/2; \,Q_1}+ N (\kappa r)^{1-(d+2)/p + \delta} [a]_{t, \delta/2; \,Q_1} \norm{D u}_{L_{p}(Q_1)}\\
&\le N(\kappa r)^{1+\hd} \big([\vec f]_{t,\hd/2; \,Q_1}+ [a]_{t, \delta/2; \,Q_1} \norm{Du}_{L_{p}(Q_{1})}\big).
\end{align*}
Also, similar to \eqref{eq2.17}, we get
\[
\abs{\tu^{\kappa r}-\tT^0_{t_0}\tu^{\kappa r}}_{0; \,Q_r(z_0)} \le  N \kappa^{\hd-1} r^{1+\hd} [u]_{t, (1+\hd)/2; \,Q^{(n+1)}}.
\]
Taking $p=\tT^0_{t_0}w+\tT^0_{t_0} \tu^{\kappa r}$, we have
\begin{align*}
r^{-1-\hd}\abs{u-p}_{0; \,Q_r(z_0)} &\le N \kappa^{\hd-1} [u]_{t,(1+\hd)/2; \,Q^{(n+1)}}+N\kappa^{1+\hd}\left( [\vec f]_{t,\hd/2; \,Q_1}+[a]_{t, \delta/2;\, Q_1}\norm{Du}_{L_{p}(Q_{1})}\right).
\end{align*}
On the other hand, for any $r\in ( 2^{-n-3}/\kappa,1/4)$, we have
\[
r^{-1-\hd}\,\inf_{p\in \tilde\bP_{0}} \,\abs{u-p}_{0; \,Q_r(z_0)}\le r^{-1-\hd} \abs{u}_{0; \,Q_r(z_0)}\le (2^{n+3}\kappa)^{1+\hd} \abs{u}_{0; \,Q_1},
\]
where $\tilde\bP_0$ denotes the set of zeroth-order partial polynomials in $t$.
Combining the two inequality above, we get similar to \eqref{eq2.48} that
\begin{multline*}
[u]_{t,(1+\hd)/2; \,Q^{(n)}} \le N\kappa^{\hd-1} [u]_{t,(1+\hd)/2; \,Q^{(n+1)}}\\
+N\kappa^{1+\hd} \left( [\vec f]_{t,\hd/2; \,Q_1}+[a]_{t, \delta/2; \,Q_1} \norm{Du}_{L_{p}(Q_{1})}\right)+N(2^{n+3} \kappa)^{1+\hd}\abs{u}_{0; \,Q_1}.
\end{multline*}
By choosing $\kappa$ sufficiently large and following the proof of \eqref{eq2.54}, we obtain
\[
[u]_{t,(1+\hd)/2; \,Q_{1/2}} \le N\big( [\vec f]_{t,\hd/2;\,Q_1}+[a]_{t, \delta/2; \,Q_1}\norm{Du}_{L_{p}(Q_{1})}+\abs{u}_{0;\,Q_1}\big).
\]
To estimate the last term on the right-hand side above, we note that by subtracting a constant we may assume that $\int_{Q_1}u=0$.
It then follows from a variant of the parabolic Poincar\'e inequality (cf. \cite[Lemma 4.2.1]{Kr08}) that
\[
\norm{u}_{L_{p}(Q_{1})}\le N \norm{\vec f}_{L_{1}(Q_{1})}+N\norm{Du}_{L_{p}(Q_{1})}.
\]
Using the equation, we further get
\[
\norm{u}_{\cH^1_{p}(Q_{1})}\le N \norm{\vec f}_{L_{1}(Q_{1})}+N\norm{Du}_{L_{p}(Q_{1})},
\]
which together with the parabolic Sobolev embedding theorem yields
\[
\abs{u}_{0;\,Q_1}\le N \left( \norm{\vec f}_{L_{1}(Q_{1})}+\norm{Du}_{L_{p}(Q_{1})}\right).
\]

To remove the smoothness condition that $u \in C^{(1+\hd)/2}_{t}$, we apply the mollification argument (with respect to $t$) as in the proof of Theorem 2.6 (i).
\hfill\qedsymbol

\subsection{Proof of Theorem~\ref{thm5}}
For the proof of the theorem, we need the following three lemmas, which are parabolic counterparts of Lemmas \ref{lem1} -- \ref{lem3}.
\begin{lemma}
                            \label{lem1p}
Let $p\in (1,\infty)$ be a constant.
Assume that $a=[a^{ij}]$ are continuous with respect to $(t,x^1,\ldots,x^{d-1})$ in $\bar Q_{2/3}$ with a modulus of continuity $\omega_a$ and continuous in $\bar Q_{1}\setminus Q_{2/3}$ with a modulus of continuity $\tilde \omega_a$.
Then for any $f\in L_p(Q_1)$, there is a unique strong solution $u\in W^{1,2}_p(Q_1)$ to the equation $u_t-a^{ij}D_{ij}u=f$ in $Q_1$ with the Dirichlet boundary condition $u=0$ on $\partial_p Q_1$.
Moreover, we have
\[
\norm{u}_{W^{1,2}_p(Q_1)}\le N \norm{f}_{L_p(Q_1)},
\]
where $N$ depends only on $d$, $p$, $\nu$, $\omega_a$, and $\tilde \omega_a$.
\end{lemma}

\begin{proof}
Similar to Lemma \ref{lem1}, the lemma is a consequence of the a priori interior $W^{1,2}_p$ estimates proved in \cite{KK07b, Dong12b} for parabolic equations with coefficients measurable in one spatial direction together with the classical boundary $W^{1,2}_p$ estimates for parabolic equations with continuous coefficients.
\end{proof}

The following lemma is the parabolic analogy of Lemma \ref{lem2}, which is also a consequence of the main results in \cite{Dong12b}.

\begin{lemma}
                                \label{lem2p}
Assume that $a=[a^{ij}]$ are continuous with respect to $(t,x^1,\ldots,x^{d-1})$ with a modulus of continuity $\omega_a$.
\begin{enumerate}[(i)]
\item
Let $p\in (1,\infty)$.
Assume that  $u\in W^{1,2}_p(Q_1)$ and satisfies $u_t-a^{ij}D_{ij}u=f$ in $Q_1$, where $f\in L_p(Q_1)$.
Then we have
\[
\norm{u}_{W^{1,2}_p(Q_{1/2})}\le N \left(\norm{f}_{L_p(Q_1)}+\norm{u}_{L_p(Q_1)}\right),
\]
where $N$ depends only on $d$,  $\nu$, $p$, and $\omega_a$.

\item
If in addition $f\in L_{\hp}(Q_1)$ for some $\hp \in (p,\infty)$, then we have $u\in W^{1,2}_{\hp}(Q_{1/2})$ and
\[
\norm{u}_{W^{1,2}_{\hp} (Q_{1/2})}\le N \left( \norm{f}_{L_{\hp}(Q_1)}+\norm{u}_{L_p(Q_1)} \right),
\]
where $N$ depends only on $d$, $\nu$, $p$, and $\omega_a$.
In particular, if $\hp>(d+2)/2$, it holds that
\[
[u]_{\gamma/2,\gamma;\, Q_{1/2}}\le N \left( \norm{f}_{L_{\hp}(Q_1)}+\norm{u}_{L_p(Q_1)} \right),
\]
where $\gamma=2-(d+2)/\hp$. If $\hp>d+2$, it holds that
\[
[Du]_{\gamma/2,\gamma; \,Q_{1/2}}\le N \left( \norm{f}_{L_{\hp}(Q_1)}+\norm{u}_{L_p(Q_1)} \right),
\]
where $\gamma=1-(d+2)/\hp$.
\end{enumerate}
\end{lemma}

The next lemma is the parabolic analogy of Lemma \ref{lem3}, which follows from the main results in \cite{DK09}.

\begin{lemma}
                                \label{lem3p}
Assume that $a=[a^{ij}]$ are continuous with respect to $(t,x^1,\ldots,x^{d-1})$ with a modulus of continuity $\omega_a$.
\begin{enumerate}[(i)]
\item
Let $p\in (1,\infty)$.
Assume that $u\in \cH^1_p(Q_1)$ and satisfies
$u_t-D_i(a^{ij}D_{j}u)=\dv \vec f$ in $Q_1$, where $\vec f= (f^1,\ldots,f^d) \in L_{p}(Q_1)$.
Then we have
\[
\norm{u}_{\cH^1_p(Q_{1/2})}\le N \left( \norm{\vec f}_{L_p(Q_1)}+\norm{u}_{L_p(Q_1)} \right),
\]
where $N$ depends only on $d$, $\nu$, $p$, and $\omega_a$.

\item
If in addition $\vec f\in L_{\hp}(Q_1)$ for some ${\hp}\in (p,\infty)$, then we have $u\in \cH^1_{\hp}(Q_{1/2})$ and
\[
\norm{u}_{\cH^1_{\hp}(Q_{1/2})}\le N \left( \norm{\vec f}_{L_{\hp}(Q_1)}+\norm{u}_{L_p(Q_1)}\right),
\]
where $N$ depends only on $d$, $\nu$, $p$, and $\omega_a$.
In particular, if ${\hp}>d+2$, it holds that
\[
[u]_{\gamma/2,\gamma;\, Q_{1/2}}\le N \left( \norm{\vec f}_{L_{\hp}(Q_1)}+\norm{u}_{L_p(Q_1)}\right),
\]
where $\gamma=1-(d+2)/{\hp}$.
\end{enumerate}
\end{lemma}

Now we are ready to prove the theorem.
Estimate \eqref{eq4.11pz} is proved by modifying the proof of \eqref{eq4.11pmz} using Lemmas \ref{lem1p} and \ref{lem2p} instead of Lemma \ref{lem4p} and the Krylov--Safonov estimate.
Similarly, assertion (ii) can be proved by modifying the proof of \eqref{eq4.12pm} using Lemma \ref{lem3p} instead of the De Giorgi--Nash--Moser estimate.

The proof of \eqref{eq4.11p} is slightly more involved.
We are not able to directly estimate $DD_{x'} u-(DD_{x'} u)_{Q_r(z_0)}$ as in the proof of Theorem \ref{thm4} (i) due to lack of a parabolic analogue of the Poincar\'e inequality, which is needed to derive a counterpart of \eqref{eq12.31}.
Our idea is to instead estimate ($D_{x'}u=D_i u$ for $i=1,\ldots, q$)
\[
D_{x'} u-(D_{x'} u)_{Q_r(z_0)}-(x-x_0)\cdot (DD_{x'} u)_{Q_r(z_0)}.
\]
We take a point $z_0\in Q_{1/2}$ and $r,R\in (0,1/4)$ such that $0<r<R/8$.
Let $\hat\zeta \in C^\infty(\bar Q_1)$ be a smooth cut-off function such that $0\le \hat\zeta\le 1$ in $Q_1$, $\hat\zeta=1$ in $Q_{1/2}$, and $\hat\zeta=0$ in $Q_{1}\setminus Q_{2/3}$.
Define
\[
\hat a^{ij}(t,x)=\hat\zeta ((t-t_0)/R^2,(x-x_0)/R )\,a^{ij}(t,x_0',x'') +\left(1-\hat\zeta((t-t_0)/R^2,(x-x_0)/R)\right)\delta_{ij}.
\]
Observe that $\hat a^{ij}$ are continuous with respect to $(t,x^1,\ldots,x^{d-1})$ in $\bar Q_{2R/3}(z_0)$ and continuous in $\bar Q_{R}(z_0)\setminus Q_{2R/3}(z_0)$.
By Lemma~\ref{lem1p}, there is a unique solution $w\in W^{1,2}_p(Q_{R}(z_0))$ of the equation
\[
w_t-\hat a^{ij}D_{ij}w=f(t,x)-f(t,x_0',x'')- (a^{ij}(t,x_0',x'')-a^{ij}(t,x))D_{ij}u
\]
in $Q_{R}(z_0)$ with zero Dirichlet boundary value on $\partial_pQ_{R}(z_0)$.
As in the proof of Theorem \ref{thm4} (i), we have
\begin{align}
\norm{w_t}_{L_p(Q_R(z_0))}&+\norm{D^2 w}_{L_p(Q_R(z_0))}\nonumber\\
&\le N \norm{f(t,x)-f(t,x_0',x'')
-(a^{ij}(t,x_0',x'')-a^{ij}(t,x))D_{ij}u}_{L_p(Q_R(z_0))}\nonumber\\
                            \label{eq11.41p}
&\le NR^{\hd+(d+2)/p}[f]_{x',\hd; \,Q_1}
+NR^{\delta}  [a]_{x', \delta; \,Q_1} \norm{D^2 u}_{L_p(Q_R(z_0))}
\end{align}
with a constant $N=N(d, p, \nu, \omega_a)$ that is independent of $R \in (0, 1/4)$.
By \eqref{eq11.41p} and \cite[Lemma 4.2.2]{Kr08}, we obtain
\begin{align}
\norm{Dw-(Dw)_{Q_R(z_0)}&-(x-x_0)\cdot(D^2 w)_{Q_R(z_0)}}_{L_p(Q_R(z_0))}\nonumber\\
&\le \norm{Dw-(Dw)_{Q_R(z_0)}}_{L_p(Q_R(z_0))}+\norm{(x-x_0)\cdot(D^2 w)_{Q_R(z_0)}}_{L_p(Q_R(z_0))}\nonumber\\
&\le NR\norm{w_t}_{L_p(Q_R(z_0))}+NR\norm{D^2 w}_{L_p(Q_R(z_0))}\nonumber\\
                            \label{eq3.43}
&\le NR^{1+\hd+(d+2)/p}[f]_{x',\hd; \,Q_1}
+NR^{1+\delta}  [a]_{x', \delta; \,Q_1} \norm{D^2 u}_{L_p(Q_R(z_0))}
\end{align}
It is easily seen that $v:=u-w\in  W^{1,2}_p(Q_{R}(z_0))$ satisfies
\begin{equation}
                            \label{eq11.47p}
v_t-a^{ij}(t,x_0',x'')D_{ij}v=f(t,x_0',x'')\quad \text{in}\quad Q_{R/2}(z_0).
\end{equation}
Note that both $a^{ij}(t,x_0',x'')$ and $f(t,x_0',x'')$ are independent of $x'$.
By mollification with respect to $x'$, without loss of generality, we may assume that
$\hat v:=D_{x'}v\in W^{1,2}_{p}(Q_{R/4}(z_0))$.
By differentiating \eqref{eq11.47p} with respect to $x'$, we see that $\hat v$ satisfies
\[
\hat v_t-a^{ij}(t,x_0',x'')D_{ij}\hat v=0\quad \text{in}\quad Q_{R/4}(z_0).
\]
Clearly, the equation above still holds with
\[
\tilde v:=\hat v-(\hat v)_{Q_{R/4}(z_0)}-(x^i-x_0^i)(D_i \hat v)_{Q_{R/4}(z_0)}
\]
in place of $\hat v$.
Also, note that
\[
V:=\hat v-(\hat v)_{Q_r(z_0)}-(x^i-x_0^i)(D_i \hat v)_{Q_{r}(z_0)}
=\tilde v-(\tilde v)_{Q_r(z_0)}-(x^i-x_0^i)(D_i \tilde v)_{Q_{r}(z_0)}
\]
satisfies the same equation and we have
\[
(V)_{Q_r(z_0)}=0,\quad DV= D\tilde v - (D \tilde v)_{Q_{r}(z_0)}.
\]
Therefore, by \cite[Lemma 4.2.1]{Kr08} applied to $V$, we have
\[
I:=\int_{Q_r(z_0)}   \Abs{\hat v-(\hat v)_{Q_r(z_0)}-(x^i-x_0^i)(D_i \hat v)_{Q_{r}(z_0)}}^p \,dxdt \le
 Nr^{p}\int_{Q_{r}(z_0)} \abs{D \tilde v-(D \tilde v)_{Q_{r}(z_0)}}^p \,dxdt.
\]
Take any $\gamma\in (\hd,1)$.
By applying Lemma \ref{lem2p} (ii) with a scaling (the modulus of continuity only improves!) to $\tilde v$, we derive from the above inequality that
\begin{align}
                                \label{eq12.31p}
I & \le Nr^{d+2+p(1+\gamma)} [D \tilde v]_{\gamma/2, \gamma; \,Q_{R/8}(z_0)}^p
\le N\left(\frac{r}{R}\right)^{d+2+p(1+\gamma)}\int_{Q_{R/4}(z_0)} \Abs{\tilde v}^p \,dxdt\nonumber\\
&= N\left(\frac{r}{R}\right)^{d+2+p(1+\gamma)}\int_{Q_{R/4}(z_0)} \Abs{\hat v-(\hat v)_{Q_{R/4}(z_0)}-(x^i-x_0^i)(D_i \hat v)_{Q_{R/4}(z_0)}}^p \,dxdt.
\end{align}
By \eqref{eq3.43}, \eqref{eq12.31p}, and the triangle inequality, we reach
\begin{align}
                                        \label{eq12.43p}
\int_{Q_r(z_0)} &\Abs{D_{x'}u-(D_{x'}u)_{Q_r(z_0)}-(x^i-x_0^i)(D_i D_{x'} u)_{Q_{r}(z_0)}}^p \nonumber\\
&\le N\left(\frac{r}{R}\right)^{d+2+p(1+\gamma)}\int_{Q_{R/4}(z_0)}\Abs{D_{x'} u-(D_{x'} u)_{Q_{R/4}(z_0)}-(x^i-x_0^i)(D_i D_{x'}u)_{Q_{R/4}(z_0)}}^p \nonumber\\
&\qquad+NR^{d+2+p(1+\hd)}\left( [f]^p_{x',\hd; \,Q_1} + [a]_{x', \delta; \,Q_1}^p   \norm{D^2u}^p_{L_p(Q_1)}\right).
\end{align}
By Lemma~\ref{lem:giaq}, we infer from \eqref{eq12.43p} that, for all $0<r<1/32$ we have
\begin{align*}
\int_{Q_r(z_0)}& \Abs{D_{x'}u-(D_{x'}u)_{Q_r(z_0)}-(x^i-x_0^i)(D_i D_{x'} u)_{Q_{r}(z_0)}}^p \nonumber\\
&\le Nr^{d+2+p(1+\hd)} \int_{Q_{1/16}(z_0)}\Abs{D_{x'}u-(D_{x'}u)_{Q_r(z_0)}-(x^i-x_0^i)(D_i D_{x'} u)_{Q_{1/16}(z_0)}}^p\\
&\qquad + Nr^{d+2+p(1+\hd)}\left([f]^p_{x',\hd; \,Q_1} + [a]_{x', \delta; \,Q_1}^p \norm{D^2u}^p_{L_p(Q_1)}\right)\nonumber\\
&\le Nr^{d+2+p(1+\hd)}\left([f]^p_{x',\hd; \,Q_1}
+ \left(1+[a]_{x', \delta; \,Q_1}^p\right) \norm{D^2u}^p_{L_p(Q_1)}\right).
\end{align*}
Therefore, we obtain \eqref{eq4.11p} by Campanato's theorem.
The theorem is proved.
\hfill\qedsymbol

\mysection{Appendix}

The proof of Corollary \ref{cor4.5} uses the following special type of interpolation inequalities for parabolic H\"older semi-norms.
For the sake of completeness and future references, we give a sketched proof.
Recall that $\bR^{d+1}_0:=(-\infty, 0) \times \bR^d$.
\begin{lemma}
                        \label{lem6.1}
Let $R\in (0,\infty)$, $\delta\in (0,1)$, and $u\in L_\infty(Q_R)$. Then we have
\begin{gather}
                        \label{eq12.16}
[u]_{z',(1+\delta)/2,1+\delta; \,Q_{R/2}} \le N\left([u]_{x',1+\delta; \,Q_R}+
[u]_{t,(1+\delta)/2; \,Q_R}\right),\\
                        \label{eq12.17}
[u]_{z',1+\delta/2,2+\delta; \,Q_{R/2}} \le N\left([u]_{x',2+\delta; \,Q_R}+
[u]_{t,1+\delta/2; \,Q_R}+R^{-2-\delta} \abs{u}_{0;\,Q_R}\right),
\end{gather}
where $N=N(d,\delta)$.
If in addition we assume $u\in L_\infty(\bR^{d+1}_0)$, then
\begin{equation}
                    \label{eq9.18}
[u]_{z',1+\delta/2,2+\delta; \,\bR^{d+1}_0}\le N\left([u]_{x',2+\delta; \,\bR^{d+1}_0}+
[u]_{t,1+\delta/2; \,\bR^{d+1}_0}\right).
\end{equation}
\end{lemma}

\begin{proof}
We use the method of finite-difference approximations.
By scaling and mollifications, without loss of generality, we may assume that $R=1$ and $u$ is infinitely differentiable in $Q_1$ with bounded derivatives.
For any unit vector $\ell\in \bR^d$ and $r>0$, we define
\[
\Delta_{\ell,r}u(t,x)=u(t,x+r\ell)-u(t,x).
\]
Also, for any $r>0$, we define
\[
\delta_{t,r}u(t,x)=u(t,x)-u(t-r,x).
\]

To show \eqref{eq12.16}, it then suffices to prove
\begin{equation}
                            \label{eq12.30}
[D_{x'}u]_{t,\delta/2;\,Q_1}\le N\left([u]_{x',1+\delta;\,Q_1}+[u]_{t,(1+\delta)/2; \,Q_1}\right).
\end{equation}
For any $(t_1,x),(t_2,x)\in Q_{1/2}$ such that $-1<t_2<t_1<0$, we denote
\[
r=\sqrt{t_1-t_2}/2\in (0,1/2).
\]
Using Taylor's formula, for any unit vector $\ell=(\ell',\ell'')\in \bR^d$ such that $\ell''=0$, we have
\[
\abs{D_{\ell}u(t_j,x)-\Delta_{\ell,r}u(t_j,x)/r}\le r^\delta[D_{x'}u]_{x',\delta; \,Q_1}
\]
for $j=1,2$, and
\begin{multline*}
\abs{\Delta_{\ell,r}u(t_1,x)/r -\Delta_{\ell,r}u(t_2,x)/r}\\
=r^{-1}\abs{(u(t_1,x+r\ell)-u(t_{2},x+r\ell))-(u(t_1,x)-u(t_{2},x))}
\le 2^{2+\delta}r^\delta[u]_{t,(1+\delta)/2; \,Q_1}.
\end{multline*}
Combining the two inequalities above, we get
\[
\abs{D_{\ell}u(t_1,x)-D_{\ell}u(t_2,x)}\le 2r^\delta[D_{x'}u]_{x',\delta;\,Q_1}+2^{2+\delta}r^\delta[u]_{t,(1+\delta)/2;\,Q_1},
\]
which yields \eqref{eq12.30}.

For \eqref{eq12.17}, it suffices to show
\begin{multline}
                    \label{eq1.37}
[D_{x'}^2u]_{t,\delta/2;\,Q_{1/2}}+[u_t]_{x',\delta;\,Q_{1/2}}+[D_{x'}u]_{t,(1+\delta)/2;\,Q_{1/2}} \\
\le N\left([u]_{x',2+\delta;\,Q_1}+[u]_{t,1+\delta/2;\,Q_1}+\abs{u}_{0;\,Q_1}\right).
\end{multline}
Let $(t,x)\in Q_{1/2}$, $r\in (0,1/8)$, and $\ell=(\ell',\ell'')\in \bR^d$ be a unit vector such that $\ell''=0$.
By using Taylor's formula, we have
\begin{multline*}
\abs{\delta_{t,r^2}^2 D_{\ell}^2u(t,x)-r^{-2}\delta_{t,r^2}^2\Delta_{\ell,r}^2u(t,x)}
=\abs{\delta_{t,r^2}^2(D_{\ell}^2u(t,x)-r^{-2}\Delta_{\ell,r}^2u(t,x))}\\
\le 4\abs{D_{\ell}^2u-r^{-2}\Delta_{\ell,r}^2u}_{0;\,Q_{1/2}} \le 4r^\delta[D_{x'}^2u]_{x',\delta;\,Q_1}
\end{multline*}
and
\[
r^{-2}\abs{\delta_{t,r^2}^2\Delta_{\ell,r}^2u(t,x)} =r^{-2}\abs{\Delta_{\ell,r}^2\delta_{t,r^2}^2u(t,x)}
\le 4r^{-2}\abs{\delta_{t,r^2}^2 u}_{0;\,Q_{1/2}}
\le 4r^\delta[u_t]_{t,\delta/2;\,Q_1}.
\]
Combining the two inequalities above, we get
\begin{equation}
                                \label{eq9.47}
\abs{\delta_{t,r^2}^2D_{\ell}^2u(t,x)}\le 4r^\delta\left( [D_{x'}^2u]_{x',\delta;\,Q_1}+[u_t]_{t,\delta/2;\,Q_1} \right).
\end{equation}
Using the simple identity
\[
2\delta_{t,r^2}=\delta_{t,2r^2}-\delta^2_{t,2r^2},
\]
we deduce from \eqref{eq9.47} by an iteration that
\begin{align}
                                \label{eq9.49}
[D_{x'}^2u]_{t,\delta/2;\,Q_{1/2}}&\le N\left( [D_{x'}^2u]_{x',\delta;\,Q_1}+[u_t]_{t,\delta/2;\,Q_1}+\abs{D_{x'}^2 u}_{0;\,Q_{3/4}} \right)\nonumber\\
&\le N\left([D_{x'}^2u]_{x',\delta;\,Q_1}+[u_t]_{t,\delta/2;\,Q_1}
+\abs{u}_{0;\,Q_{1}}\right),
\end{align}
where in the last inequality we used the classical interpolation inequality with respect to $x'$.
Similarly, we have
\[
\abs{\Delta_{\ell,r}^3u_t(t,x)-r^{-2}\Delta_{\ell,r}^3 \delta_{t,r^2}u(t,x)}
\le 8r^\delta[u_t]_{t,\delta/2;\,Q_1}
\]
and
\begin{align*}
r^{-2}\abs{\Delta_{\ell,r}^3\delta_{t,r^2}u(t,x)}=r^{-2}\abs{\delta_{t,r^2}\Delta_{\ell,r}^3u(t,x)}
\le 2r^\delta[D_{x'}^2u]_{x',\delta;\,Q_1},
\end{align*}
which yields
\begin{equation}
                        \label{eq10.04}
\abs{\Delta_{\ell,r}^3u_t(t,x)}
\le 8r^\delta[u_t]_{t,\delta/2;\,Q_1}+2r^\delta[D_{x'}^2u]_{x',\delta;\,Q_1}.
\end{equation}
Using the identity
\[
3\Delta_{\ell,r}u(t,x)=\Delta_{\ell,3r}u(t,x-\ell r)-\Delta^3_{\ell,r}u(t,x-\ell r),
\]
we deduce from \eqref{eq10.04} by an iteration that
\begin{align}
                                \label{eq10.12}
[u_t]_{x,\delta;\,Q_{1/2}}&\le N\left([D_{x'}^2u]_{x',\delta;\,Q_1}+[u_t]_{t,\delta/2;\,Q_1}
+\abs{u_t}_{0;\,Q_{3/4}}\right)\nonumber\\
&\le N\left([D_{x'}^2u]_{x',\delta;\,Q_1}+[u_t]_{t,\delta/2;\,Q_1}
+\abs{u}_{0;\,Q_{3/4}}\right).
\end{align}
where in the last inequality we used the classical interpolation inequality with respect to $t$.
Moreover, combining
\[
\abs{\delta_{t,r^2}^2 D_{\ell}u(t,x)
-(2r)^{-1}\delta^2_{t,r^2}\Delta_{\ell,2r}u(t,x-r\ell)}\le Nr^{1+\delta}[D_{x'}^2u]_{x',\delta;\,Q_1}
\]
and
\[
r^{-1}\abs{\delta^2_{t,r^2}\Delta_{\ell,2r}u(t,x-r\ell)}
=r^{-1}\abs{\Delta_{\ell,2r}\delta^2_{t,r^2} u(t,x-r\ell)}\le Nr^{1+\delta}[u_t]_{t,\delta/2;\,Q_1},
\]
we get
\begin{equation}
                                \label{eq9.47b}
\abs{\delta_{t,r^2}^2 D_{\ell}u(t,x)}\le Nr^{1+\delta}\left([D_{x'}^2u]_{x',\delta;\,Q_1}+[u_t]_{t,\delta/2;\,Q_1}\right).
\end{equation}
Similar to \eqref{eq9.49} and \eqref{eq10.12}, we deduce from \eqref{eq9.47b} that
\begin{align}
                               \label{eq11.20}
[D_{x'}u]_{t,(1+\delta)/2;\,Q_{1/2}}&\le N\left([D_{x'}^2u]_{x',\delta;\,Q_1}+[u_t]_{t,\delta/2;\,Q_1}+
\abs{D_{x'}u}_{0;\,Q_{3/4}}\right)\nonumber\\
&\le N\left([D_{x'}^2u]_{x',\delta;\,Q_1}+[u_t]_{t,\delta/2;\,Q_1}+
\abs{u}_{0;\,Q_{1}}\right).
\end{align}
From \eqref{eq9.49}, \eqref{eq10.12}, and \eqref{eq11.20}, we reach \eqref{eq1.37}, and thus \eqref{eq12.17}.
Finally, \eqref{eq9.18} follows from \eqref{eq12.17} by sending $R\to \infty$.
\end{proof}

\begin{remark}
From the proof it is easily seen that \eqref{eq12.16} still holds when $\delta=1$.
In general, estimate \eqref{eq12.17} is not true without the lower-order term $R^{-2-\delta} \abs{u}_{0;\,Q_1}$ on the right-hand side.
For instance, if $u=\abs{x'}^2 t$, then we have $[u]_{x',2+\delta;\,Q_R}=[u]_{t,1+\delta/2;\,Q_R}=0$, but $[u]_{z',1+\delta/2,2+\delta;\,Q_{R/2}}>0$.
\end{remark}

As we are not able to find an explicit reference in the literature, we also include the following version of Campanato's theorem for $p\in (0,1)$\footnote{We would like to thank Prof. Mikhail Safonov for the simplified proof of this lemma.}.
We only present the elliptic case, since the parabolic case is similar.

\begin{lemma}
                    \label{lemma6.14}
Let $p\in (0,1)$, $R_0\in (0,1/2)$, $\delta\in (0,1]$, and $u\in L_{p;\, loc}(B_1)$ be a function. Suppose that for any $x_0\in B_{1/2}$ and $r\in (0,R_0)$, we can find a constant $c_{x_0,r}$ such that
\begin{equation}
                                \label{eq3.11}
\fint_{B_r(x_0)}\,\Abs{u-c_{x_0,r}}^p\,dx\le M^pr^{p\delta}
\end{equation}
for some constant $M>0$ independent of $x_0$ and $r$. Then there exists a function $v\in C^\delta(B_{1/2})$ such that $u=v$ in $B_{1/2}$ a.e. and
\[
[v]_{\delta;B_{1/2}}\le N(\delta,p,R_0)\,M.
\]
\end{lemma}
\begin{proof}
For any $x_0\in B_{1/2}$, $r\in (0,R_0)$, and $r'\in [r/2,r)$, by the triangle inequality,
\[
\abs{c_{x_0,r}-c_{x_0,r'}}^p\le \bigl( \abs{u(x)-c_{x_0,r'}} + \abs{u(x)-c_{x_0,r}}\bigr)^p \le \abs{u(x)-c_{x_0,r'}}^p+\abs{u(x)-c_{x_0,r}}^p
\]
holds for any $x\in B_{r'}(x_0)$.
Taking the average with respect to $x\in B_{r'}(x_0)$, we get
\begin{equation}
                            \label{eq3.02}
\abs{c_{x_0,r}-c_{x_0,r'}}\le NMr^\delta.
\end{equation}
Therefore, for fixed $x_0\in B_{1/2}$, $c_{x_0,r}$ converges as $r\to 0$. We denote the limit to be $v(x_0)$.
Then again by \eqref{eq3.02} 
\[
\abs{c_{x_0,r}-v(x_0)}\le NM r^\delta.
\]
This and \eqref{eq3.11} imply that
\begin{equation}
                                \label{eq3.06}
\fint_{B_r(x_0)} \abs{u-v(x_0)}^p\,dx\le N M^p r^{p\delta}.
\end{equation}
Now for any $x_0,y_0\in B_{1/2}$ such that $\tau:=\abs{x_0-y_0}<R_0$, we have
\[
\abs{v(x_0)-v(y_0)}^p\le \bigl( \abs{u(x)-v(x_0)}+\abs{u(x)-v(y_0)}\bigr)^p
\le \abs{u(x)-v(x_0)}^p+\abs{u(x)-v(y_0)}^p.
\]
Taking the average with respect to $x\in B_\tau(x_0)\cap B_\tau(y_0)$, we easily get from \eqref{eq3.06} that
\[
\abs{v(x_0)-v(y_0)}\le NM\tau^\delta,
\]
which implies that $v\in C^\delta(B_{1/2})$ and $[v]_{\delta;\,B_{1/2}}\le NM$.

On the other hand, because $u\in L_{p;\, loc}(B_1)$, for any fixed constant $c$ and a.e. $x_0\in B_1$, by the Lebesgue differentiation theorem we have
\[
\fint_{B_r(x_0)} \abs{u-c}^p\,dx \to \abs{u(x_0)-c}^p\quad \text{as}\quad r\to 0.
\]
Therefore, for a.e. $x_0\in B_1$ and any $c\in \bQ$, we have
\[
\fint_{B_r(x_0)} \abs{u-c}^p\,dx \to \abs{u(x_0)-c}^p\quad \text{as}\quad r\to 0.
\]
Next, for any such $x_0$, we take a sequence of rational number $\set{c_k}$ which converges to $u(x_0)$.
Since
\[
\fint_{B_r(x_0)} \abs{u-u(x_0)}^p\,dx \le \fint_{B_r(x_0)} \abs{u-c_k}^p\,dx+\abs{c_k-u(x_0)}^p,
\]
letting $k$ sufficiently large and then $r$ sufficiently small, we see that
\[
\fint_{B_r(x_0)} \abs{u-u(x_0)}^p \,dx \to 0\quad \text{as}\quad r\to 0.
\]
Therefore, $v=u$ a.e. in $B_{1/2}$ .
The lemma is proved.
\end{proof}

\begin{remark}
It is worth noting that when $p\in (0,1)$, in general the following inequalities, which may seem very plausible, actually {\em do not} hold:
\begin{gather*}
\int_{B_r} \abs{u(x)-(u)_{B_r}}^p\,dx\le N\int_{B_r} \abs{u(x)}^p\,dx,\\
\int_{B_r} \abs{u(x)-(u)_{B_r}}^p\,dx\le N\int_{B_{2r}} \abs{u(x)-(u)_{B_{2r}}}^p\,dx.
\end{gather*}
In fact, one can easily construct a function such that $\int_{B_{2r}} \abs{u(x)}^p\,dx\le 1$, $(u)_{B_{2r}}=0$, but $(u)_{B_r}$ is as large as we want.
\end{remark}

\begin{acknowledgment}
The authors would like to thank the referees for reading of the manuscript and many useful comments.
\end{acknowledgment}



\begin{thebibliography}{m}

\bibitem{ADN64} Agmon, S.; Douglis, A.; Nirenberg, L.
\textit{Estimates near the boundary for solutions of elliptic partial differential equations satisfying general boundary conditions, I},
Comm. Pure Appl. Math., \textbf{12}  (1959), 623--727; II, ibid., \textbf{17}  (1964), 35--92.

\bibitem{CaCa95} Caffarelli, L. A.; Cabr\'e, X.
\textit{Fully nonlinear elliptic equations.}
American Mathematical Society, Providence, RI, 1995.

\bibitem{CKV} Chipot, M.; Kinderlehrer, D.; Vergara-Caffarelli, G.
\textit{Smoothness of linear laminates.}
Arch. Ration. Mech. Anal. \textbf{96} (1986) no. 1, 81--96.

\bibitem{DK09} Dong, H.; Kim, D.
\textit{Parabolic and elliptic systems in divergence form with variably partially BMO coefficients},
SIAM J. Math. Anal. \textbf{43} (2011) no. 3, 1075--1098.

\bibitem{DK11} Dong, H.; Kim, S.
\textit{Partial Schauder estimates for second-order elliptic and parabolic equations}, Calc. Var. Partial Differential Equations \textbf{40} (2011), no. 3-4, 481--500.

\bibitem{DoKr10} Dong, H.; Krylov, N. V.
\textit{Second-order elliptic and parabolic equations with $B(\bR^2,\text{VMO})$ coefficients}, Trans. Amer. Math. Soc. \textbf{362} (2010), no. 12, 6477--6494.

\bibitem{DoKrLi} Dong, H.; Krylov, N. V.; Li, X.
\textit{On fully nonlinear elliptic and parabolic equations with VMO coefficients in domains}. Algebra i Analiz \textbf{24} (2012), no. 1, 53--94; translation in St. Petersburg Math. J. \textbf{24} (2013), no. 1, 39--69


\bibitem{Dong12b} Dong, H.
\textit{Solvability of second-order equations with hierarchically partially BMO coefficients}, Trans. Amer. Math. Soc. \textbf{364} (2012), no. 1, 493--517.

\bibitem{Dong12} Dong, H.
\textit{Gradient estimates for parabolic and elliptic systems from linear laminates}, Arch. Ration. Mech. Anal. \textbf{205} (2012), no. 1, 119--149.


\bibitem{Fife} Fife, P.
\textit{Schauder estimates under incomplete H\"older continuity assumptions},
Pacific J. Math. \textbf{13} (1963) 511--550.


\bibitem{Giaq83} Giaquinta, M.
\textit{Multiple integrals in the calculus of variations and nonlinear elliptic systems.}
Princeton University Press:Princeton, NJ, 1983.

\bibitem{GT} Gilbarg, D.; Trudinger, N. S.
\textit{Elliptic partial differential equations of second order.}
Reprint of the 1998 ed. Springer-Verlag, Berlin, 2001.

\bibitem{JLW14} Jin, Y.; Li, D.; Wang, X.-J.
\textit{Regularity and analyticity of solutions
in a direction for elliptic equations}, Pacific J. Math. \textbf{276} (2015), no. 2, 419--436.

\bibitem{Knerr} Knerr, B. F.
\textit{Parabolic interior Schauder estimates by the maximum principle},
Arch. Rational Mech. Anal. \textbf{75} (1980/81), no. 1, 51--58.


\bibitem{Kr96} Krylov, N. V.
\textit{Lectures on elliptic and parabolic equations in H\"older spaces.}
American Mathematical Society, Providence, RI, 1996.

\bibitem{Kr08} Krylov, N. V.
\textit{Lectures on elliptic and parabolic
 equations in Sobolev spaces.} American Mathematical Society, Providence,
RI, 2008.

\bibitem{KK07} Kim, D.; Krylov, N. V.
\textit{Elliptic differential equations with coefficients measurable with respect to one variable and VMO with respect to the others}, SIAM J. Math. Anal. \textbf{39} (2007), no. 2, 489--506.

\bibitem{KK07b} Kim, D.; Krylov, N. V.
\textit{Parabolic equations with measurable coefficients}, Potential Anal. \textbf{26} (2007), no. 4, 345--361.

\bibitem{Kr10} Krylov, N. V.
\textit{On Bellman's equations with VMO coefficients},
Methods Appl. Anal. \textbf{17} (2010), no. 1, 105--121.

\bibitem{KrPr} Krylov, N. V.; Priola, E.
\textit{Elliptic and parabolic second-order PDEs with growing coefficients},
Comm. Partial Differential Equations. \textbf{35} (2010) 1532--4133.

\bibitem{LV00} Li, Y.Y.; Vogelius, M.
\textit{Gradient estimates for solutions to divergence form elliptic
equations with discontinuous coefficients},
Arch. Ration. Mech. Anal. \textbf{153} (2000), no. 2, 91--151.

\bibitem{LN03} Li, Y.Y.; Nirenberg, L.
\textit{Estimates for elliptic systems from composite material},
Comm. Pure Appl. Math. \textbf{56} (2003), no. 7, 892--925.

\bibitem{Li86} Lin, F.-H.
\textit{Second derivative $L^p$-estimates for elliptic equations of nondivergent type}, Proc. Amer. Math. Soc. \textbf{96} (1986), no. 3, 447--451.

\bibitem{Lieb92} Lieberman, G. M.
\textit{Intermediate Schauder theory for second order parabolic equations. IV. Time irregularity and regularity},
Differential Integral Equations \textbf{5} (1992), no. 6, 1219--1236.

\bibitem{Lieberman} Lieberman, G. M.
\textit{Second order parabolic differential equations.}
World Scientific Publishing Co., Inc., River Edge, NJ, 1996.

\bibitem{Lorenzi} Lorenzi, L.
\textit{Optimal Schauder estimates for parabolic problems with data measurable with respect to time},
SIAM J. Math. Anal. \textbf{32} (2000), no. 3, 588--615.

\bibitem{TW} Tian, G.; Wang, X.-J.
\textit{Partial regularity for elliptic equations}, Discret. Contin. Dyn. Syst. \textbf{28} (2010), no. 3, 899--913.

\bibitem{Trudinger} Trudinger, N. S.
\textit{A new approach to the Schauder estimates for linear elliptic equations},
Proc. Centre Math. Anal. Austral. Nat. Univ. \textbf{14},  52--59,
Austral. Nat. Univ., Canberra, 1986.

\end{thebibliography}
\end{document}